%% file: ExpectedShortfall.tex
\documentclass[a4paper,oneside,11pt]{article}
\usepackage{a4wide} 
\usepackage{graphicx}
\usepackage[latin1]{inputenc}
\usepackage{amsmath,amssymb,amsthm}
\usepackage{url}
\usepackage{ifthen}
\usepackage[sectionbib]{natbib}
\usepackage{chapterbib}
\usepackage{color}
\usepackage{epstopdf}
\usepackage{caption}
\usepackage{subcaption}
\usepackage{units}
\usepackage{enumitem}
\usepackage{pdfpages}
\usepackage{titling}

\usepackage{array}
\usepackage{hhline}
\newcolumntype{C}[1]{>{\centering}p{#1}}  
\newcolumntype{L}[1]{>{\raggedleft}p{#1}} 
\newcolumntype{R}[1]{>{\raggedright}p{#1}} 

\newcolumntype{Z}[1]{>{\centering\arraybackslash}m{#1}} 
\newcolumntype{I}[1]{>{\raggedleft\arraybackslash}m{#1}} 
\newcolumntype{E}[1]{>{\raggedright\arraybackslash}m{#1}} 

\theoremstyle{plain}
\newtheorem{theorem}{Theorem}%
\newtheorem{prop}[theorem]{Proposition}%
\newtheorem{lemma}[theorem]{Lemma}%
\theoremstyle{definition}
\newtheorem{example}{Example}%
\newtheorem*{remark*}{Remark}%
\renewcommand\theta{\vartheta}

\newcommand{\N}{\ensuremath{\mathbb{N}}}%
\newcommand{\Z}{\ensuremath{\mathbb{Z}}}%
\newcommand{\R}{\ensuremath{\mathbb{R}}}%
\newcommand{\dif}{\mathrm{d}}

\DeclareMathOperator{\Var}{Var}
\DeclareMathOperator{\Cov}{Cov}

\DeclareMathOperator*{\argmin}{arg\,min}

\renewcommand\theta{\vartheta}

\newboolean{comment}
\newcommand{\Kommentar}[1]{%
        \ifthenelse{\boolean{comment}}{#1}{}
        }
\setboolean{comment}{true}

\begin{document}

{\title{Asymptotics for the expected shortfall}}

\include{MainPart}

\setcounter{section}{0}
\setcounter{equation}{0}
\setcounter{page}{1}

\include{supplement}

\end{document}

%% file: MainPart.tex
\author{Tobias Zwingmann\protect\footnote{Email: zwingmann@mathematik.uni-marburg.de} \protect\ and Hajo Holzmann\protect\footnote{Corresponding author, Email: holzmann@mathematik.uni-marburg.de}  \protect\\ \protect\small{Fachbereich Mathematik und Informatik}  \protect\\ \protect\small{Philipps-Universität Marburg, Germany} \protect\\}
\maketitle

\begin{abstract}
	We derive the joint asymptotic distribution of empirical quantiles and expected shortfalls under general conditions on the distribution of the underlying observations. In particular, we do not assume that the distribution function is differentiable at the quantile with strictly positive derivative. Hence the rate of convergence and the asymptotic distribution for the quantile can be non-standard, but our results show that the expected shortfall remains asymptotically normal with a $\sqrt{n}$-rate, and we even give the joint distribution in such non-standard cases. In the derivation we use the bivariate scoring functions for quantile and expected shortfall as recently introduced by \citet{fisszieg2015elicitability}. The main technical issue is to deal with the distinct rates for quantile and expected shortfall when applying the argmax-continuity theorem. We also consider spectral risk measures with finitely-supported spectral measures, and illustrate our results in a simulation study.
\end{abstract}

\section{Introduction}
Value-at-risk (VaR) and expected shortfall (ES) are two popular measures of the risk of a financial position \citep{mcneil2010quantitative}. While the VaR is simply a quantile of the profit-and-loss-distribution, the ES is defined as the average below a certain quantile. Thus, the ES is deemed to be more informative, and indeed fulfills the desirable property of subadditivity which the VaR lacks in general \citep{artzner1999}. On the other hand, the VaR is elicitable \citep{gneiting2011} in the sense that it can be represented as a minimizer of an expected loss, which is, however, not possible for the ES.  

Statistical estimation of a given $\alpha$-quantile, $\alpha \in (0,1)$, is a very-well developed problem. Precise asymptotic expansions, called Bahardur expansions, for the empirical quantile have been developed if the underlying distribution function $F$ has a density which is positive and sufficiently regular at the $\alpha$-quantile \citep{bahadur,kiefer}. This expansion in particular implies the asymptotic normality. In this regular case, an alternative quantile estimator based on a smoothed empirical distribution function has been proposed by \citet{chenTang2005} to improve finite-sample mean-square-error properties. The general case in which the distribution function $F$ is not differentiable at the $\alpha$-quantile or in which its derivative vanishes was studied in \citet{smirnov1952limit} and \citet{Knight2002LimDistr}. Here, non-normal limit distributions and slower rates of convergence than $\sqrt{n}$ occur.  

The ES at level $\alpha$ can be estimated as the empirical average below the empirical $\alpha$-quantile. \citet{scaillet2004} proposed instead to use a smoothed version of this estimator. \citet{chen2008} proved asymptotic normality of these estimators and further showed that no improvement in terms of mean-square-error properties can be expected for the smoothed estimator. Further work on the asymptotic properties of ES estiamtors are \citet{linton2013} and \citet{hill2013} for heavy-tailed distributions, and \citet{peracchi2008}, \citet{taylor2008}, \citet{cai2008} and \citet{kato2012} in a nonparametric regression framework. 

All these papers require that the distribution function is quite regular at its $\alpha$-quantile, having a smooth and positive density as required for asymptotic normality when estimating the quantile. 

Here we show that this assumptions is not required for the expected shortfall, and that the simple estimator of the ES remains normal under the weak assumption that the distribution function is continuous and strictly increasing at its $\alpha$-quantile. We even determine the joint asymptotic distribution of the estimators for $\alpha$-quantile and expected shortfall in this general case. Our approach is based on the argmax-continuity theorem, e.g.~\citet{vdv2000asympstat}, by using the scoring functions for the bivariate parameter, quantile and ES, as recently introduced by \citet{fisszieg2015elicitability}. Because of the different rates of convergence for quantile and ES, application of the argmax-continuity theorem is not straightforward and requires substantial technical effort. 

The paper is organised as follows. In Section \ref{sec:basicmethods} we introduce the expected shortfall and the bivariate scoring function for quantile and expected shortfall, and discuss the resulting minimum-contrast estimators. 
Section \ref{sec:estconsist} presents our results on the joint asymptotic distribution of quantile and expected shortfall, also in the multivariate case for various levels. We further discuss asymptotic properties of estimators of spectral risk measures with finitely supported spectral measures. 
Section \ref{sec:sims} contains simulations in two scenarios, once for a kink in the distribution function, and once for a density with a zero of order two. 
In Section \ref{sec:discuss} we indicate properties of the bootstrap, and also extensions to dependent data. Proofs of the major steps are deferred to Section \ref{sec:proofs}, while some details are relegated to the technical supplement \citet{zwingHolzm2016}.

In the rest of the paper, we use the following notation.
For i.i.d.~observations $X_1, \ldots, X_n$ distributed according to the distribution function $F$, we use the notation \begin{equation*}
\mathbb{E}_n [f(Y)]  =\frac{1}{n}\sum_{i=1}^{n} f(X_i) \qquad \text{ and } \quad \mathbb{G}_n [f(Y)] = \sqrt{n}(\mathbb{E}_n-\mathbb{E})[f(Y)]
\end{equation*}
where $E |f(X_1)| < \infty$. Note that $\mathbb{E}_n$ is the empirical expectation w.r.t. $X_1, \ldots , X_n$. 
Further, we let $X_{1:n} \leq \ldots \leq X_{n:n}$ denote the order statistics of a sample $X_1, \ldots, X_n$. 
We denote by $\Rightarrow$ convergence in distribution. 

%
\section{Estimating quantile and expected shortfall}\label{sec:basicmethods}
Suppose that the random variable $Y$ has distribution function $F$ and satisfies $\mathbb{E}[ |Y|] < \infty$. Given $ \alpha \in (0,1)$ the lower tail expected shortfall of $Y$ at level $\alpha$ is defined by
\[ es_\alpha = \frac{1}{\alpha}\, \int_0^\alpha F^{-1}(u)\, du.\]
For the specific value of $\alpha$ under consideration, we shall always impose the following. 
\medskip

{\bf Assumption.}\quad For the given $\alpha \in (0,1)$, the distribution function $F$ is continuous and strictly increasing at its $\alpha$-quantile $q_\alpha$.

\medskip

Then $F$ has a unique $\alpha$-quantile, and the empirical quantile
\begin{equation*}
\overline{q}_{n,\alpha} = \inf\Big\{x \Big| \sum_{i=1}^{n}1_{X_i\leq x} \geq n\alpha \Big\} = X_{\lceil n\,\alpha\rceil:n},
\end{equation*}
is a consistent estimator for $q_\alpha$. Further, for the expected shortfall we have that

\begin{equation}\label{eq:lowertailes}
es_\alpha = \frac{1}{\alpha}\, \int_{- \infty}^{q_\alpha}\, y\, dF(y).
\end{equation}
Consider the class of strictly consistent scoring functions for the bivariate parameter $(q_\alpha, es_\alpha)$ as introduced by \citet{fisszieg2015elicitability},
\begin{align}\label{eq:s_rep_1}
\begin{split}
S(x_1, x_2; y) & = \big(1_{ y \leq x_1} - \alpha \big) (x_1-y) + G(x_2) \, \big(x_2 + \alpha^{-1}\big(1_{ y \leq x_1}-\alpha\big) (x_1 - y)\,   \big)\\
& \qquad - \mathcal{G}(x_2) - G(x_2)y,
\end{split}
\end{align}
where $\mathcal{G}$ is a three-times continuously differentiable function, $\mathcal{G}' = G$ and it is required that $G' >0$. From the proof of Corollary 5.5 in \citet{fisszieg2015elicitability}, one may choose $\mathcal{G}$ so that $\lim_{x\to -\infty}G(x)=0$. 
Denote the asymptotic contrast function	by
\[	S(x_1, x_2; F) = \mathbb{E} [S(x_1, x_2; Y)],\]
then $S(x_1,x_2;F)$ has a unique minimum in $(q_{\alpha}, es_{\alpha})$. 

Let $X_1, \ldots , X_n$ be i.i.d., distributed according to $F$ with $ \mathbb{E} |X_1| < \infty$. Consider the minimum contrast estimator for the bivariate parameter $(q_\alpha, es_\alpha)$ defined by
\begin{align*}
(\widehat{q}_{n,\alpha}, \widehat{es}_{n,\alpha}) &\in \argmin_{(x_1,x_2)\in\R^2} \sum_{i=1}^{n} S(x_1,x_2;X_i).
\end{align*}

As the proposition below shows, this is, at least approximately, simply another way of representing standard estimators for quantile and expected shortfall. 

\begin{prop}\label{lem:connect_to_emp_vers}
	The estimator $\widehat{q}_{n,\alpha}$ can be chosen equal to the empirical quantile. Further, the estimator $\widehat{es}_{n,\alpha}$ is given by 
	\begin{align}\label{eq:esscoringexpl}
	\begin{split}		
	\widehat{es}_{n,\alpha} & = \argmin_{x_2\in\R}\, \sum_{i=1}^{n} S\big(\widehat{q}_{n,\alpha},x_2;X_i\big)\\
	& = \alpha^{-1}\mathbb{E}_n\big(Y \,1_{Y\leq \widehat{q}_{n,\alpha}}\big) + \widehat{q}_{n,\alpha} \, \big(1 - \frac{1}{\alpha\, n}\sum_{i=1}^{n}1_{X_i\leq \widehat{q}_{n,\alpha}} \big),
	\end{split}
	\end{align}
	and we have that
	\[		\Big|\widehat{es}_{n, \alpha}-\alpha^{-1}\mathbb{E}_n\big(Y \,1_{Y\leq \widehat{q}_{n,\alpha}}\big)\Big|\leq \frac{\widehat{q}_{n,\alpha}}{\alpha \,n}=O_{\mathcal{P}}(n^{-1}).
	\]	
\end{prop}

The {\sl empirical expected shortfall} may be defined to be
\begin{equation*}
\alpha^{-1}\mathbb{E}_n\big(Y \,1_{Y\leq \widehat{q}_{n,\alpha}}\big)=	\frac{1}{\alpha\,n}\sum_{i=1}^{n}1_{X_i\leq \widehat{q}_{n,\alpha}} X_i  . 
\end{equation*}
As the proposition shows, the estimator $\widehat{es}_{n,\alpha}$ is, up to a term of order $O_{\mathcal{P}}(n^{-1})$, equal to this empirical expected shortfall. Thus, its asymptotic properties will be identical to those of $\widehat{es}_{n,\alpha}$.

\section{Joint asymptotic theory for quantile and expected shortfall}\label{sec:estconsist}

We start the asymptotic analysis by providing a general consistency result. 

\begin{prop}\label{thm:emp_vers_cons}
	Let $q_n$ be a consistent estimator of $q_{\alpha}$. Then the estimators 
	\begin{equation*}
	\alpha^{-1} \mathbb{E}_n\big[Y \,1_{Y\leq q_{n}}\big] \quad\text{and}\quad 	\widetilde{es}_{n, \alpha} = \argmin_{x_2\in\R} \sum_{i=1}^{n} S(q_{n}, x_2; X_i).
	\end{equation*}
	are consistent for $es_{\alpha}$. In particular, $\big(\widehat{q}_{n,\alpha},\widehat{es}_{n, \alpha}\big)$ is consistent. 	
\end{prop}

Now we turn to the joint asymptotic distribution of quantile and expected shortfall.

One major issue is to include the case of low regularity of $F$ at its $\alpha$-quantile $q_\alpha$. In particular, we do not impose the standard assumption that $F$ has a positive derivative at $q_\alpha$. 
In such more general situations, the limit distribution for the empirical quantile has been analyzed in \citet{smirnov1952limit} and \citet{Knight2002LimDistr}. 

Consider the following assumption, taken from \citet{smirnov1952limit} and \citet{Knight2002LimDistr}. 

\medskip

\textbf{Assumption [A]}: There exists a function $\psi_{\alpha}:\R\to\overline{\R}$ with 
\begin{align*}
\lim_{t\to\infty} \psi_{\alpha}(t)=\infty, \qquad 
\lim_{t\to -\infty}\psi_{\alpha}(t)=-\infty
\end{align*}
such that for some deterministic, positive sequence $(a_n)_n$ with $a_n \to \infty$ it holds that \begin{equation*}
\lim_{n\to\infty} \sqrt{n}\big[F(q_{\alpha}+t/a_n)-F(q_{\alpha}) \big] = \psi_{\alpha}(t).
\end{equation*}

\medskip

The following proposition, which is mainly taken from \citet{smirnov1952limit}, recalls the classification of the functions  $\psi_{\alpha}$ which may occur in Assumption [A] and further shows that, if the empirical $\alpha$-quantile is a consistent estimator for $q_\alpha$, then Assumption [A] can always be satisfied with a degenerate choice for the function $\psi_{\alpha}$. 
\begin{prop}\label{prop:theassumptionA}
	a.~The function $\psi_{\alpha}$ in Assumption [A] is necessarily of one of the forms
	\begin{align*}
	\psi_{\alpha}(t) &= \begin{cases}
	\kappa_+ t^{\beta} &\mbox{ if } t\geq 0\\
	-\infty &\mbox{ if } t<0, 
	\end{cases}\qquad \quad 
	\psi_{\alpha}(t) = \begin{cases}
	-\kappa_- (-t)^{\beta} &\mbox{ if } t\leq 0\\
	\infty &\mbox{ if } t>0, 
	\end{cases}\\
	\psi_{\alpha}(t) &= \begin{cases}
	-\kappa_- (-t)^{\beta} &\mbox{ if } t\leq 0\\
	\kappa_+t^{\beta} &\mbox{ if } t>0, 
	\end{cases}\qquad 
	\psi_{\alpha}(t) = \begin{cases}
	-\infty &\mbox{ if } t<-c_1 \\
	0 &\mbox{ if } -c_1 \leq t \leq c_2\\
	\infty &\mbox{ if } t>c_2, 
	\end{cases}
	\end{align*}
	where $\kappa_+ , \kappa_- , \beta>0$ and $c_1,c_2\geq 0$. Moreover, except for the last case with $c_1 = c_2 = 0$, i.e.~$\psi_\alpha(t) = \infty \,\cdot \text{sign}\,(t)$ with $\infty \cdot 0 = 0$, the sequence $a_n$ is uniquely determined up to asymptotic equivalence. \\
	b. If the empirical $\alpha$-quantile is consistent, so that $(X_{\lceil n\,\alpha \rceil:n}-q_{\alpha})=o_{\mathcal{P}}(1)$, then there exists a sequence $(a_n)$ for which Assumption [A] is satisfied for the limit function $\psi_\alpha(t) = \infty \,\cdot \text{sign}(t)$. 
\end{prop}
Here, sequences of positive numbers $(a_n)$ and $(b_n)$ are asymptotically equivalent if $a_n / b_n \to 1$, $n \to \infty$. Part b.~of the proposition implies that Assumption [A] imposes no additional general restrictions if $F$ is strictly increasing and continuous at its $\alpha$-quantile.

\begin{example}\label{ex:expansion_around_quantile}

	Assume that there exists an $\varepsilon>0$ and functions $\kappa_+,\,\kappa_-$, which are continuous in $q_{\alpha}$ with $\kappa_+(q_{\alpha}),\, \kappa_-(q_{\alpha})\ne 0$ and fulfill
	\begin{align*}
	F(x)-\alpha &= (x-q_{\alpha})^{r+1}\,\kappa_+(x), \quad \text{for } x\in[q_{\alpha}, q_{\alpha}+\varepsilon),\\
	F(x)-\alpha&= (q_{\alpha}-x)^{l+1}\,\kappa_-(x), \quad \text{for } x\in(q_{\alpha}-\varepsilon, q_{\alpha}]
	\end{align*}
	for some $r,l\in(-1,\infty)$. 
	For example, if $F$ has a density with a root of order $k\in \N_0$ in its $\alpha$-quantile, these assertions are met; see Sections \ref{subsec:jump_disc} and \ref{subsec:root_order_two}. 
	
	Since we assume strict monotonicity of $F$ in $q_\alpha$, we must have $\kappa_+(x)>0$ for $x\in(q_{\alpha}, q_{\alpha}+\varepsilon)$ and hence $\kappa_+(q_\alpha) >0$ as well (it is $\not=0$ by assumption), and similarly $\kappa_-(x)<0$ for $x\in(q_{\alpha}-\varepsilon, q_{\alpha}]$. Then for $t>0$, setting $a_n^r=n^{\nicefrac{1}{2(r+1)}}$ we have that 
	$\nicefrac{t}{a_n^r}\in [q_{\alpha}, q_{\alpha}+\varepsilon)$ for $n$ big enough, hence
	\begin{equation}\label{eq:conv_right_side_correct}
	\sqrt{n}\Big(F\big(q_{\alpha}+t/a_n^r\big)-F\big(q_{\alpha}\big)\Big) = \sqrt{n}\,\kappa_+\big(q_{\alpha}+t/a_n^r\big) \,\frac{t^{r+1}}{\sqrt{n}} \rightarrow \kappa_+(q_{\alpha})t^{r+1} >0, \quad n \to \infty.  
	\end{equation}
	Similarly for $t<0$ and $a_n^l = n^{\nicefrac{1}{2(l+1)}}$ we have that
	\begin{equation}\label{eq:conv_left_side_correct}
	\sqrt{n}\Big(F\big(q_{\alpha}+t/a_n^l\big)-F\big(q_{\alpha}\big)\Big) \longrightarrow \kappa_-(q_{\alpha})(-t)^{l+1}<0.
	\end{equation}
	
	Now if $r=l$, we can choose $a_n = n^{\nicefrac{1}{2(r+1)}}$ and 
	\begin{equation*}
	\psi_{\alpha}(t) = \begin{cases}
	\kappa_-\big(q_{\alpha}\big)\,(-t)^{r+1}, &\mbox{ if } t\leq 0,\\
	\kappa_+\big(q_{\alpha}\big)\,t^{r+1}, &\mbox{ if } t> 0.
	\end{cases}
	\end{equation*}
	Then the sequence $(a_n)$ together with the function $\psi_{\alpha}$ fulfill Assumption [A]. 
	
	If $r>l$, for $t<0$ choosing $a_n = n^{\nicefrac{1}{2(r+1)}}$ we have for $n$ big enough that 
	\begin{equation}\label{eq:conv_left_side_wrong}
	\sqrt{n}\Big(F\big(q_{\alpha}+t/a_n^r\big)-F\big(q_{\alpha}\big)\Big) = n^{\nicefrac{r-l}{2(r+1)}}\,\kappa_-\big(q_{\alpha}+t/a_n^r\big) \, (-t)^{l+1} \longrightarrow -\infty.
	\end{equation}
	Thus, Assumption [A] is then satisfied for $a_n = n^{\nicefrac{1}{2(r+1)}}$ and \begin{equation*}
	\psi_{\alpha}(t) = \begin{cases}
	-\infty, &\mbox{ if } t< 0,\\
	\kappa_+\big(q_{\alpha}\big)\,t^{r+1}, &\mbox{ if } t\geq 0.
	\end{cases}
	\end{equation*}
	
	The case $l>r$ is treated similarly. 
	
\end{example}

The second assumption will guarantee the existence of a limit variance for the estimator $\widehat{es}_{n,\alpha}$. 

\medskip

\textbf{Assumption [B]}: It holds that $\mathbb{E} \big[1_{Y \leq 0}\, Y^2\big]<\infty$.

\medskip

Now we may state our main result. 

\begin{theorem}\label{thm:main_thm_es}
	
	Under Assumptions [A] and [B], we have that 
	\[ \big(a_n(\widehat{q}_{n,\alpha}-q_{\alpha}), \sqrt{n}(\widehat{es}_{n,\alpha}-es_{\alpha})\big) \Rightarrow \big(\psi^{\leftrightarrow}_{\alpha}( W_1), W_2 \big),\]
	where $(W_1,W_2)$ are jointly normally distributed,  
	\begin{equation*}
	(W_1,W_2)\sim \mathcal{N}(0,\Sigma), \qquad 	\Sigma=\begin{pmatrix}
	\alpha(1-\alpha) && (1-\alpha)(q_{\alpha}-es_{\alpha})\\
	(1-\alpha)(q_{\alpha}-es_{\alpha}) && \Var\big(1_{Y\leq q_{\alpha}}(q_{\alpha}-Y)/\alpha\big)
	\end{pmatrix}
	\end{equation*}
	and
	\begin{equation}\label{eq:formulainvfunct}
	\psi_{\alpha}^{\leftrightarrow}(x)=\begin{cases}
	\inf\{t\leq 0 \,|\, \psi_{\alpha}(t)\geq x\} &\mbox{if } x<0\\
	0 &\mbox{if } x=0\\
	\sup\{t\geq 0 \,|\, \psi_{\alpha}(t)\leq x\} &\mbox{if } x>0.
	\end{cases}
	\end{equation}
	
\end{theorem}

\begin{remark*}
	The theorem implies that the marginal asymptotic distribution of the estimator $\widehat{es}_{n,\alpha}$ is not effected by low regularity of the distribution function $F$ at $q_\alpha$, although rate of convergence and asymptotic distribution of $\widehat{q}_{n,\alpha}$ become non-standard. This is not unsurprising, for the following reason. For a known value $q_\alpha$ of the $\alpha$-quantile, one could consider the oracle estimator 
	\[  \frac{1}{\alpha\,n}\sum_{i=1}^{n}1_{X_i\leq {q}_{\alpha}} X_i,\]
	which has asymptotic variance $\Var\big(1_{Y\leq q_{\alpha}}\, Y/\alpha\big)$. Now 
	\[ \Var\big(1_{Y\leq q_{\alpha}}(q_{\alpha}-Y)/\alpha\big) - \Var\big(1_{Y\leq q_{\alpha}}\, Y/\alpha\big) = \frac{1-\alpha}{\alpha}\, q_\alpha\, \big(q_\alpha - 2 es_\alpha \big).\]
	If $q_\alpha<0$, which is plausible in applications since we consider the lower-tail expected shortfall, then since $es_\alpha \leq q_\alpha$ the difference will be negative, so that estimating the quantile actually may reduce the asymptotic variance of the expected shortfall. This effect persists even it is quite hard -- as in situations with low regularity of $F$ at $q_\alpha$ -- to estimate the quantile. 
\end{remark*}

\begin{remark*}
	Chen and Tang (2005) proposed a smoothed estimator of the quantile and showed that higher-order correction of the MSE is possible for appropriate choice of the bandwidth. Scaillet (2004) proposed a smoothed estimator of the expected shortfall, but the asymptotic analysis in Chen (2008) showed that no asymptotic improvement can be expected, thus, Chen (2008) recommends the use of the simple empirical expected shortfall. 
	What is more, the favourable analysis of Chen and Tang (2005) for the smoothed estimator of the quantile depends on regularity of $F$ and $q_\alpha$, roughly a twice-continuously differentiable density.  We shall investigate behaviour of the smoothed estimator of the expected shortfall in our less regular situations in the simulation study. 
\end{remark*}

\begin{remark*}
	The proof of Theorem \ref{thm:main_thm_es} relies on the argmax-continuity theorem as presented e.g.~in \citet{vdv2000asympstat}, Corollary 5.58. However, this cannot be applied directly since the contrast process does not properly converge when normalized with a single rate, and the proof becomes quite involved. 
\end{remark*}

\begin{example}[{\sl Example \ref{ex:expansion_around_quantile} continued}]
	Consider the situation of Example \ref{ex:expansion_around_quantile},  and additionally assume that Assumption [B] is satisfied. 
	If $r=l$, Theorem \ref{thm:main_thm_es} applies with $a_n = n^{\nicefrac{1}{2(r+1)}}$ and 
	\begin{equation*}
	\psi^{\leftrightarrow}_{\alpha}(u) = \begin{cases}
	-\big(u/\kappa_-(q_{\alpha})\big)^{\nicefrac{1}{l+1}}, &\mbox{ if } u<0,\\
	0, &\mbox{ if } u=0,\\
	\big(u/\kappa_+(q_{\alpha})\big)^{\nicefrac{1}{r+1}}, &\mbox{ if } u>0.
	\end{cases}
	\end{equation*}
	For $r>l$ Theorem \ref{thm:main_thm_es} still applies with $a_n = n^{\nicefrac{1}{2(r+1)}}$, but in the formula for $\psi^{\leftrightarrow}_{\alpha}(u)$ as above, we have to replace the case $u<0$ with $\psi^{\leftrightarrow}_{\alpha}(u)=0$.
\end{example}
Next let us extend Theorem \ref{thm:main_thm_es} to a multivariate version. 
For given $k$ choose distinct $\alpha_m\in(0,1)$, $m\in \{1, \ldots , k\}$, and assume as before that $F$ is strictly monotone and continuous at each quantile $q_{\alpha_m}$.\\ 

\textbf{Assumption [A$^k$]:} For each $m\in\{1, \ldots, k\}$ and corresponding $\alpha_m$ and $q_{\alpha_m}$, Assumption [A] is satisfied with associated sequence $(a_{m,n})_n$ and function $\psi_{\alpha_m}(t)$.

\begin{theorem}\label{cor:conv_estimator_mult}
	Let [A$^k$] and [B] hold. 	
	Then as $n \to \infty$, 
	\begin{align*}
	\Big(a_{1,n} \big( \widehat{q}_{n,\alpha_1} - q_{\alpha_1}\big), \sqrt{n}\, \big(\widehat{es}_{n,\alpha_1} - es_{\alpha_1}\big),&\ldots , a_{k,n} \big( \widehat{q}_{n,\alpha_k} - q_{\alpha_k}\big), \sqrt{n}\, \big(\widehat{es}_{n,\alpha_k} - es_{\alpha_k}\big)\Big) \\[5pt]
	&\Rightarrow (z_{1,1}, z_{1,2}, \ldots , z_{k,1}, z_{k,2}),
	\end{align*}
	where for $m=1, \ldots, k$, 
	$		z_{m,1} = \psi^{\leftrightarrow}_{\alpha_m}( W{m,1}) $ and $ 
	z_{m,2} = W_{m,2}$,
	with $\psi^{\leftrightarrow}_{\alpha_m}$ as in (\ref{eq:formulainvfunct})	and the vector $\big(W_{1,1},W_{1,2}, \ldots , W_{k,1},W_{k,2}\big)$ distributed according to $\mathcal{N}\big(0,\Sigma\big)$ with $\Sigma$ determined by \begin{align*}
	\Cov \big(W_1^s,W_1^t\big) &= \alpha_s\wedge \alpha_t - \alpha_s \alpha_t,\\[5pt]
	\Cov \big(W_2^s, W_2^t) &= \frac{\alpha_s\wedge\alpha_t}{\alpha_s\alpha_t}\big(q_{\alpha_s}q_{\alpha_t} - (q_{\alpha_s}+q_{\alpha_t})es_{\alpha_s\wedge\alpha_t} \big) + \frac{1}{\alpha_s\alpha_t} \mathbb{E}\big[1_{Y\leq q_{\alpha_s\wedge \alpha_t}}Y^2\big],\\
	& \quad + \big(es_{\alpha_s} - q_{\alpha_s} \big) \big(es_{\alpha_t} - q_{\alpha_t} \big)\\[5pt]
	\Cov(W_1^s,W_2^t) &= \frac{\alpha_s\wedge\alpha_t}{\alpha_t}\big(q_{\alpha_t}-es_{\alpha_s\wedge \alpha_t}\big) - \alpha_s\big(q_{\alpha_t}-es_{\alpha_t}\big)
	\end{align*}
	for $s,t\in\{1,\ldots , k\}$. 	
\end{theorem}

The extension of the proof of Theorem \ref{thm:main_thm_es} to the multivariate case in Theorem \ref{cor:conv_estimator_mult} is relegated to the technical supplement.  


As an application of the above theorem, consider estimation of spectral risk measures with finite support. 
For a probability measure $\mu$ on $[0,1]$, called the \emph{spectral measure}, define
\begin{equation*}
\nu_{\mu}(F) = \int_{[0,1]} es_{\alpha} \,\dif \mu(\alpha)
\end{equation*}
as the \emph{spectral risk measure} associated to $\mu$. Here, the boundary cases are given by $es_1 = E [Y]$ and $es_0 =$ essinf$\, Y$. 
If $\mu$ is finitely supported in $(0,1)$, $\nu_{\mu}(F)$ is a finite convex combination of expected shortfalls for different levels,
\begin{equation*}
\nu_{\mu} = \sum_{m=1}^{k}p_m \, es_{\alpha_m} \quad \text{if} \quad \mu = \sum_{m=1}^{k} p_m \delta_{\alpha_m}.
\end{equation*}
\citet{fisszieg2015elicitability} show that strictly consistent scoring functions for $\nu_{\mu}$ in this case are given by
\begin{align*}
S_{sp}(x_1, \ldots , x_k, x_{k+1};z) &= \sum_{m=1}^{k} \Bigg(\Big(1+\frac{p_m}{\alpha_m}G(x_{k+1})\Big)\big(1_{z\leq x}-\alpha\big)\,(x-z) \\
&\qquad\qquad + p_m\big(G(x_{k+1})(x_{k+1}-z) - \mathcal{G}(x_{k+1}\big)\big)\Bigg). 
\end{align*}
where the functions $G$ and $\mathcal{G}$ are as above. 
If we define the corresponding M-estimator
\[
\big(\widehat{q}_{n,\alpha_1}, \ldots , \widehat{q}_{n,\alpha_k}, \widehat{\nu}_{\mu,n}\big) \in \argmin_{x_1, \ldots , x_{k+1}}\, \sum_{i=1}^{n} S_{sp}(x_1, \ldots , x_{k+1}; Y_i),
\]
then we have the following result. 
\begin{theorem}\label{th:estspectralrisk}
	We have that 
	\begin{equation}\label{eq:specrisk}
	\widehat{\nu}_{\mu,n} = \sum_{m=1}^{k}p_m\, \widehat{es}_{n,\alpha_m}.
	\end{equation}
	Consequently, under Assumptions [A$^k$] and [B] it follows that
	\[ \sqrt{n}\, \big(\widehat{\nu}_{\mu,n} - \widehat{\nu}_{\mu} \big) \Rightarrow \sum_{m=1}^k p_m \, W_{m,2},\] 
	where the $W_{m,2}$ are as in Theorem \ref{cor:conv_estimator_mult}. 
\end{theorem}
The first part of the theorem is proved similarly to Proposition \ref{lem:connect_to_emp_vers}, details are given in the technical supplement. The second part then follows directly from formula (\ref{eq:specrisk}) and Theorem \ref{cor:conv_estimator_mult}. 


\section{Simulations}\label{sec:sims}

\subsection{Distribution function with kink in the $\alpha$-quantile}\label{subsec:jump_disc}

We let 
$F$ be given by 
\begin{equation*}
F(x) = \frac{1}{5}(x+1)\, 1_{(-1,0]}(x) + \big(\frac{1}{5}+\frac{8}{5}\,x \big)\, 1_{(0,\nicefrac{1}{2}]}(x) + 1_{(\nicefrac{1}{2}, \infty)}(x).
\end{equation*}
Then $q_{\nicefrac{1}{5}}=0$, and the expected shortfall at level $1/5$ is $es_{\nicefrac{1}{5}} = -\nicefrac{1}{2}$. Further, $F'(0-)=\nicefrac{1}{5}$ and $F'(0+)=\nicefrac{8}{5}$, where $F'(0\pm)$ denote the left- and right-sided derivatives in $0$. 

Taylor expansion shows that Example~\ref{ex:expansion_around_quantile} applies with $r=l=0$, $a_n=\sqrt{n}$ and \begin{equation*}
\psi_{\nicefrac{1}{5}}(t) =	t\, \big( F'(0-)\,1_{t\leq 0} +F'(0+)\,1_{t\geq 0}\big)
\end{equation*}
so that \begin{equation*}
\psi_{\nicefrac{1}{5}}^{\leftrightarrow}(t)= t\,\big(1_{t\leq 0}/F'(0-) + 1_{t\geq 0}/F'(0+)\big).
\end{equation*}
It follows that 
\begin{equation*}
\sqrt{n}\, \big(\widehat{q}_{n,\nicefrac{1}{5}}-0,\,\widehat{es}_{n,\nicefrac{1}{5}}+\nicefrac{1}{2} \big)\Rightarrow \Big(W_1 \big(1_{W_1 < 0}/F'(0-) + 1_{W_1 > 0}/F'(0+)\big) , W_2\Big),
\end{equation*}
where \begin{equation*}
(W_1,W_2)  \sim \mathcal{N}\left(0,\begin{pmatrix} \nicefrac{4}{25} & \nicefrac{2}{5} \\  \nicefrac{2}{5}  & \nicefrac{5}{3}
\end{pmatrix} \right).
\end{equation*}
The limit distribution function of $\sqrt{n}\,\widehat{q}_{n,\nicefrac{1}{5}}$ is calculated as \begin{equation*}
z\mapsto\Phi_{0,\nicefrac{4}{25}}\big(z\,(1_{z< 0}\,F'(0-) + 1_{z> 0}\,F'(0+))\big).
\end{equation*}
We compute the estimators for simulated samples of sizes $n \in \{10^2,10^3 ,10^4, 5\, 10^4, 10^5, 10^6\}$, each for $m=5\cdot 10^3$ iterations, using the using \textbf{R} programming language.  
Figure \ref{fig:root_order_two_cdf_1} shows estimated and asymptotic distribution functions of  
\[n^{\nicefrac{1}{2}}\,(\widehat{q}_{n,\nicefrac{1}{5}}) \quad \text{and} \quad n^{\nicefrac{1}{2}}\,(\widehat{es}_{n,\nicefrac{1}{5}}+\nicefrac{1}{2}),\]
for samples of sizes $n \in \{n^2, n^3, n^4\}$. The approximation is reasonable in both cases also for small sample sizes. 

\begin{figure}
	\centering
	\begin{subfigure}{.5\textwidth}
		\centering
		\includegraphics[width=\linewidth]{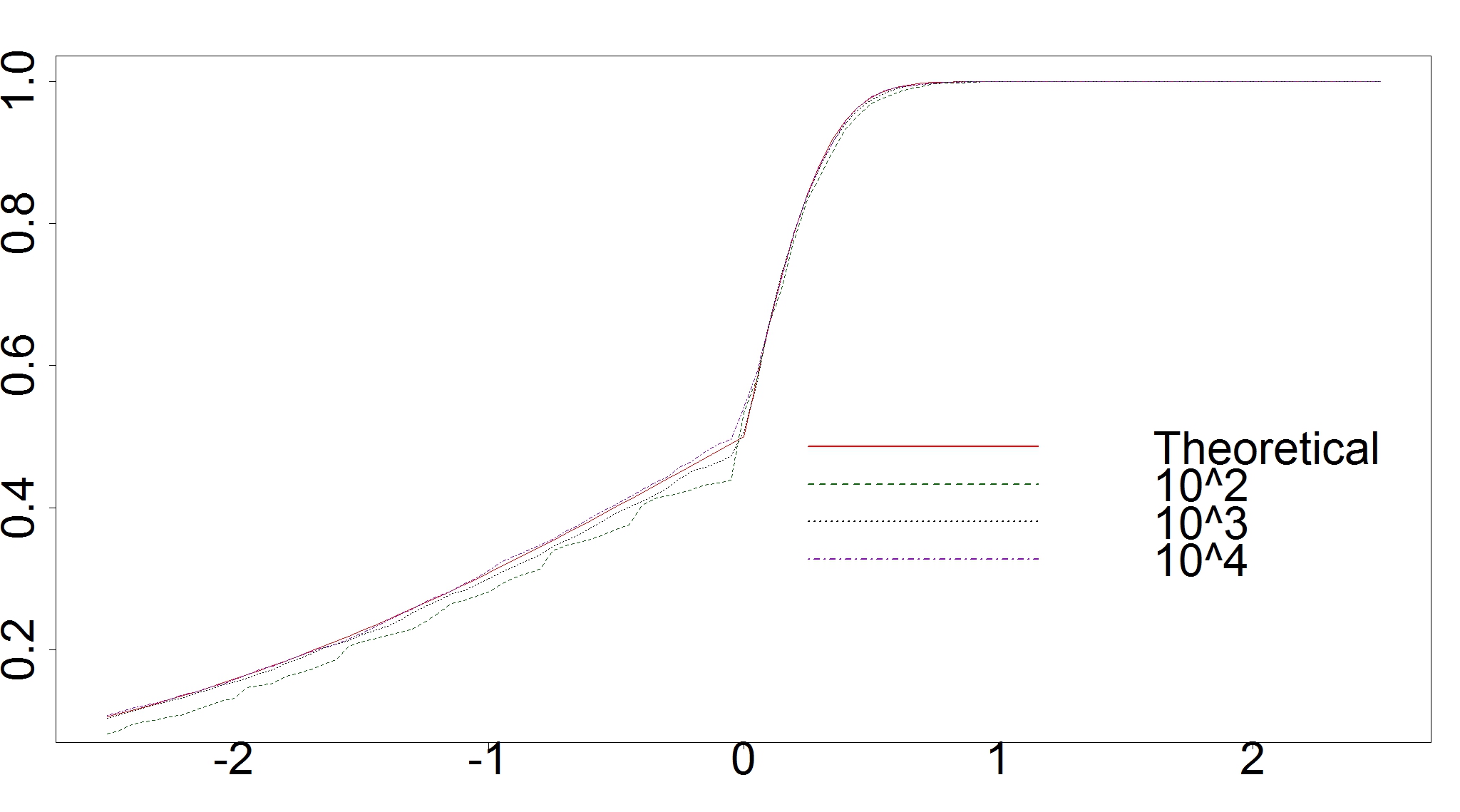}
		\caption{Quantile estimator}
		\label{subfig:quantile_sprung_1}
	\end{subfigure}%
	\begin{subfigure}{.5\textwidth}
		\centering
		\includegraphics[width=\linewidth]{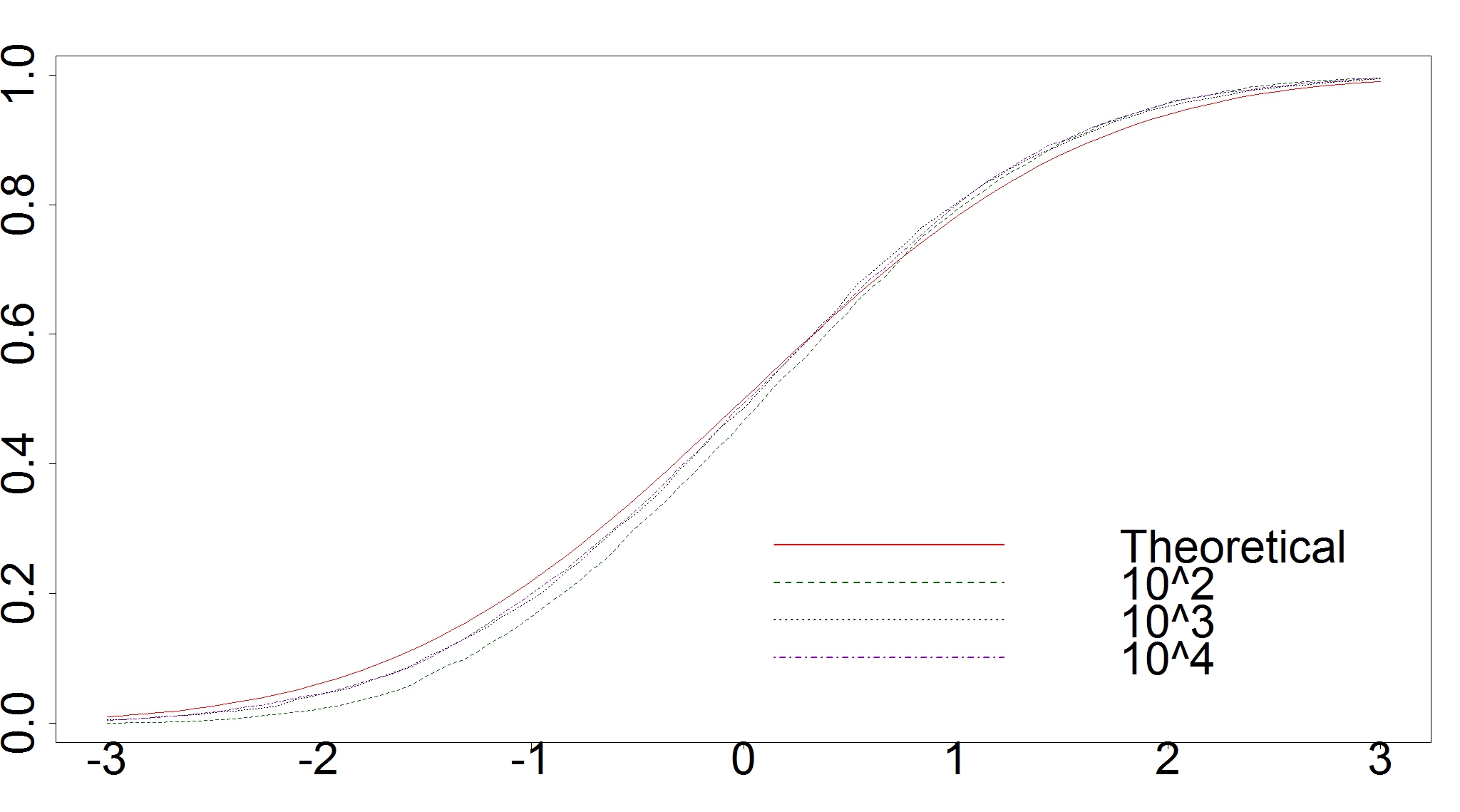}
		\caption{Expected Shortfall estimator}
		\label{subfig:expectedshort_sprung_1}
	\end{subfigure}
	\captionsetup{font={small}, labelfont={bf}, width=.95\linewidth}
	\caption[ Example~\ref{subsec:jump_disc}; Asymptotics of $\big(\widehat{q}_{n,\nicefrac{1}{5}},\widehat{es}_{n,\nicefrac{1}{5}}\big)$ for $m=5\cdot 10^3$]{The left picture (a) shows the limit (red solid) and the estimated distribution function of $n^{\nicefrac{1}{2}}\big(\widehat{q}_{n,\nicefrac{1}{5}}\big)$ for the distribution function of Example~\ref{subsec:jump_disc} for $n=10^2$ (green dashed), $n=10^3$ (black dotted) and $n=10^4$ (purple dot-dashed). \\ The right picture (b) accordingly shows the estimated and the limit (solid red) distribution function of $\sqrt{n}\big(\widehat{es}_{n,\nicefrac{1}{5}}+\nicefrac{1}{2}\big)$ for $n=10^2$ (green dashed), $n=10^3$ (black dotted) and $n=10^4$ (purple dot-dashed).\\ Here $m=5\cdot 10^3$ was chosen for both estimations.}
	\label{fig:jump_1}
\end{figure}

From the same data we in addition computed the smoothed quantile estimator $\widetilde{q}_{h,n,\nicefrac{1}{5}}$ and the estimator $\widetilde{es}_{h,n,\nicefrac{1}{5}}$ for the expected shortfall as proposed in \cite{chenTang2005} and \citet{chen2008}, respectively. Here we used fixed bandwidths $h_n$ chosen as the median normal reference bandwidth of additional training samples. 

Table~\ref{tab:bias_var_4.1} shows the mean and the standard deviation of the centered and rescaled estimators for the quantile and the expected shortfall, as well as their correlation. 
We observe that the limit distribution of $\sqrt{n}\, \widehat{q}_{n,\nicefrac{1}{5}}$ does not have mean $0$, while the mean of $\sqrt{n}\, \widetilde{q}_{n,\nicefrac{1}{5}}$ seems to diverge. Smoothing the expected shortfall also seems to introduce a small bias, without substantially reducing the standard deviation.

\renewcommand{\arraystretch}{1.1}

\begin{table}[htb]
	\centering
	\captionsetup{font={small}, labelfont={bf}, width=0.9\linewidth}
	\caption[Example~\ref{subsec:jump_disc}:  Sample bias and standard deviation]{Sample bias and standard deviation of the rescaled and centred estimators calculated over $5000$ estimates with $h\in\{0.082, 0.052, 0.0333, 0.024, 0.0201, 0.0132\}$ }\label{tab:bias_var_4.1}
	\begin{tabular}{|E{2cm}|E{3.2cm}|*6{Z{30pt}|}}
		\hhline{|========|}
		Samplesize & $n$ & $10^2$ & $10^3$ & $10^4$ & $5\cdot 10^4$ & $10^5$ & $10^6$ \tabularnewline 
		\hhline{|========|} 
		Mean & $\sqrt{n}\,\widehat{q}_{n,\nicefrac{1}{5}}$& $ -0.59 $ & $ -0.67 $ & $ -0.71 $ & $ -0.69 $ & $ -0.7 $ & $ -0.75 $ \tabularnewline 
		& $\sqrt{n}\,\widetilde{q}_{h,n,\nicefrac{1}{5}}$& $ -0.91 $ & $ -1.5 $ & $ -2.81 $ & $ -4.43 $ & $ -5.23 $ & $ -10.81 $
		\tabularnewline 
		\hhline{|~-------|}
		&  $\sqrt{n}\big(\widehat{es}_{n,\nicefrac{1}{5}}+\nicefrac{1}{2}\big)$& $ 0.11 $ & $ 0.02 $ & $ 0.01 $ & $ 0 $ & $ 0 $ & $ 0 $    \tabularnewline  
		& $\sqrt{n}\big(\widetilde{es}_{h,n,\nicefrac{1}{5}}+\nicefrac{1}{2}\big)$& $ 0.19 $ & $ 0.12 $ & $ 0.15 $ & $ 0.16 $ & $ 0.16 $ & $ 0.21 $   \tabularnewline 
		\hhline{|========|}
		Standard & $\sqrt{n}\,\widehat{q}_{n,\nicefrac{1}{5}}$& $ 1.12 $ & $ 1.22 $ & $ 1.24 $ & $ 1.23 $ & $ 1.21 $ & $ 1.22 $ \tabularnewline
		deviation& $\sqrt{n}\,\widetilde{q}_{h,n,\nicefrac{1}{5}}$& $ 0.95 $ & $ 0.91 $ & $ 0.84 $ & $ 0.81 $ & $ 0.8 $ & $ 0.81 $ \tabularnewline
		\hhline{|~-------|}
		& $\sqrt{n}\big(\widehat{es}_{n,\nicefrac{1}{5}}+\nicefrac{1}{2}\big)$ & $ 1.1 $ & $ 1.18 $ & $ 1.18 $ & $ 1.19 $ & $ 1.19 $ & $ 1.19 $ \tabularnewline
		&$\sqrt{n}\big(\widetilde{es}_{h,n,\nicefrac{1}{5}}+\nicefrac{1}{2}\big)$ & $ 1.12 $ & $ 1.2 $ & $ 1.19 $ & $ 1.2 $ & $ 1.19 $ & $ 1.19 $\tabularnewline
		\hhline{|========|}
		Correlation & $\sqrt{n}\,\widehat{q}_{n,\nicefrac{1}{5}}$ and $\sqrt{n}\big(\widehat{es}_{n,\nicefrac{1}{5}}+\nicefrac{1}{2}\big)$ & $ 0.76 $ & $ 0.76 $ & $ 0.76 $ & $ 0.76 $ & $ 0.76 $ & $ 0.76 $
		\tabularnewline
		\hhline{|========|}
	\end{tabular} 
\end{table}
\renewcommand{\arraystretch}{1}

\subsection{Density with root of order $2$}\label{subsec:root_order_two}

Let $\alpha = \nicefrac{1}{2}$ and 
\[ 
F(x) = 1_{[0,2]}(x)\, 
\frac{((x-1)^3+1)}{2} + 1_{(2,\infty)}(x).
\]
Then $F(1)=\nicefrac{1}{2}$, so that $q_{\nicefrac{1}{2}}=1$ and $es_{\nicefrac{1}{2}}=\nicefrac{1}{4}$. Example~\ref{ex:expansion_around_quantile} applies with $r=l=2$, $\varepsilon=1$ and $\kappa_+(x)=-\kappa_-(x)=1/2$, hence $a_n=n^{\nicefrac{1}{6}}$ and $\psi_{\nicefrac{1}{2}}(t)=t^3/2$ satisfy Assumption [A]. The map $\psi_{\nicefrac{1}{2}}$ is invertible with $\psi_{\nicefrac{1}{2}}^{\leftrightarrow}(y) = \psi_{\nicefrac{1}{2}}^{-1}(y)=(2y)^{\nicefrac{1}{3}}$. Assumption [B] is fulfilled as well with 
$4\,\Var\big(1_{Y\leq 1}(1-Y)\big) =\nicefrac{51}{80}
$. 
Using Theorem~\ref{thm:main_thm_es} we obtain 
\begin{equation*}
\big(n^{\nicefrac{1}{6}}\,(\widehat{q}_{n,\nicefrac{1}{2}}-1), n^{\nicefrac{1}{2}}\,(\widehat{es}_{n,\nicefrac{1}{2}}-\nicefrac{1}{4})\big) \Rightarrow \big(\psi_{\nicefrac{1}{2}}^{-1}(W_1),\,W_2\big) = \big((2\,W_1)^{\nicefrac{1}{3}},\,W_2\big),
\end{equation*}
where 

\begin{equation*}
(W_1,W_2)\sim\mathcal{N}(0,\Sigma), \qquad \Sigma = \begin{pmatrix}
\nicefrac{1}{4} && \nicefrac{3}{8}\\
\nicefrac{3}{8} && \nicefrac{51}{80}
\end{pmatrix}.
\end{equation*}
The distribution function of $\psi_{\nicefrac{1}{2}}^{-1}(W_1)$ is 
$ \Phi_{0,\nicefrac{1}{4}}\big(x^3/2\big),
$
and the joint density of $\big(\psi_{\nicefrac{1}{2}}^{-1}(W_1),\,W_2\big)$ is given by

\begin{align*}
f_{\psi_{\nicefrac{1}{2}}^{-1}(W_1),W_2}(x,y) = \frac{3t^2}{4\,\pi\sqrt{\det \Sigma}} \exp\Big(-\frac{1}{8} (t^3~v)\Sigma^{-1} (t^3~v)^T\Big) \,. 
\end{align*}
We compute the estimators for simulated samples of sizes $n \in \{10^2,10^3 ,10^4, 5\, 10^4, 10^5, 10^6\}$, each for $m=5\cdot 10^3$ iterations. 

Figure \ref{fig:root_order_two_cdf_1} shows estimated and asymptotic distribution functions of  
\[n^{\nicefrac{1}{6}}\,(\widehat{q}_{n,\nicefrac{1}{2}}-1) \quad \text{and} \quad n^{\nicefrac{1}{2}}\,(\widehat{es}_{n,\nicefrac{1}{2}}-\nicefrac{1}{4}),\]
where Figure~\ref{fig:root_order_two_cdf_1} contains the results for $n \in \{10^2,10^3 ,10^4\}$ for the quantile estimator as well as $n \in \{10^2,10^3 ,10^4, 10^6\}$ for the expected shortfall estimator, and  Table~\ref{tab:bias_var_3.1} shows the means and the standard deviations as well as the correlations of the centred and rescaled estimators. \\

\renewcommand{\arraystretch}{1.1}

\begin{table}[htb]
	\centering
	\captionsetup{font={small}, labelfont={bf}, width=.9\linewidth}
	\caption[Example~\ref{subsec:root_order_two}: Sample bias and standard deviation]{Sample bias, standard deviation and correlation of the centred and rescaled estimators calculated over $5000$ estimates}\label{tab:bias_var_3.1}
	\begin{tabular}{|E{2cm}|E{3.2cm}|*6{Z{30pt}|}}
		\hhline{|========|}
		Samplesize & $n$ & $10^2$ & $10^3$ & $10^4$ & $5\cdot 10^4$ & $10^5$ & $10^6$ \tabularnewline 
		\hhline{|========|} 
		Mean & $\sqrt{n}\,\big(\widehat{q}_{n,\nicefrac{1}{2}}-1\big)$& $ -0.01 $ & $ -0.01 $ & $ 0.02 $ & $ -0.03 $ & $ 0 $ & $ 0.01 $ \tabularnewline 
		\hhline{|~-------|}
		&  $\sqrt{n}\big(\widehat{es}_{n,\nicefrac{1}{2}}-\nicefrac{1}{4}\big)$& $ 0.28 $ & $ 0.2 $ & $ 0.14 $ & $ 0.09 $ & $ 0.1 $ & $ 0.07 $    \tabularnewline  
		\hhline{|========|}
		Standard & $\sqrt{n}\,\big(\widehat{q}_{n,\nicefrac{1}{2}}-1\big)$& $ 0.88 $ & $ 0.89 $ & $ 0.89 $ & $ 0.9 $ & $ 0.9 $ & $ 0.9 $ \tabularnewline
		\hhline{|~-------|}
		deviation& $\sqrt{n}\big(\widehat{es}_{n,\nicefrac{1}{2}}-\nicefrac{1}{4}\big)$ & $ 0.83 $ & $ 0.82 $ & $ 0.79 $ & $ 0.81 $ & $ 0.8 $ & $ 0.79 $  \tabularnewline
		\hhline{|========|}
		Correlation & $\sqrt{n}\,\big(\widehat{q}_{n,\nicefrac{1}{2}}-1\big)$ and $\sqrt{n}\big(\widehat{es}_{n,\nicefrac{1}{2}}-\nicefrac{1}{4}\big)$ & $ 0.83 $ & $ 0.85 $ & $ 0.86 $ & $ 0.87 $ & $ 0.87 $ & $ 0.87 $ 
		\tabularnewline
		\hhline{|========|}
	\end{tabular} 
	
\end{table}

\renewcommand{\arraystretch}{1}

Overall, the asymptotic approximation is reasonable for the quantile already for moderate sample sizes, but the expected shortfall requires quite large sample sizes for the asymptotic approximation to become valid. 

\begin{figure}
	\centering
	\begin{subfigure}{.5\textwidth}
		\centering
		\includegraphics[width=\linewidth]{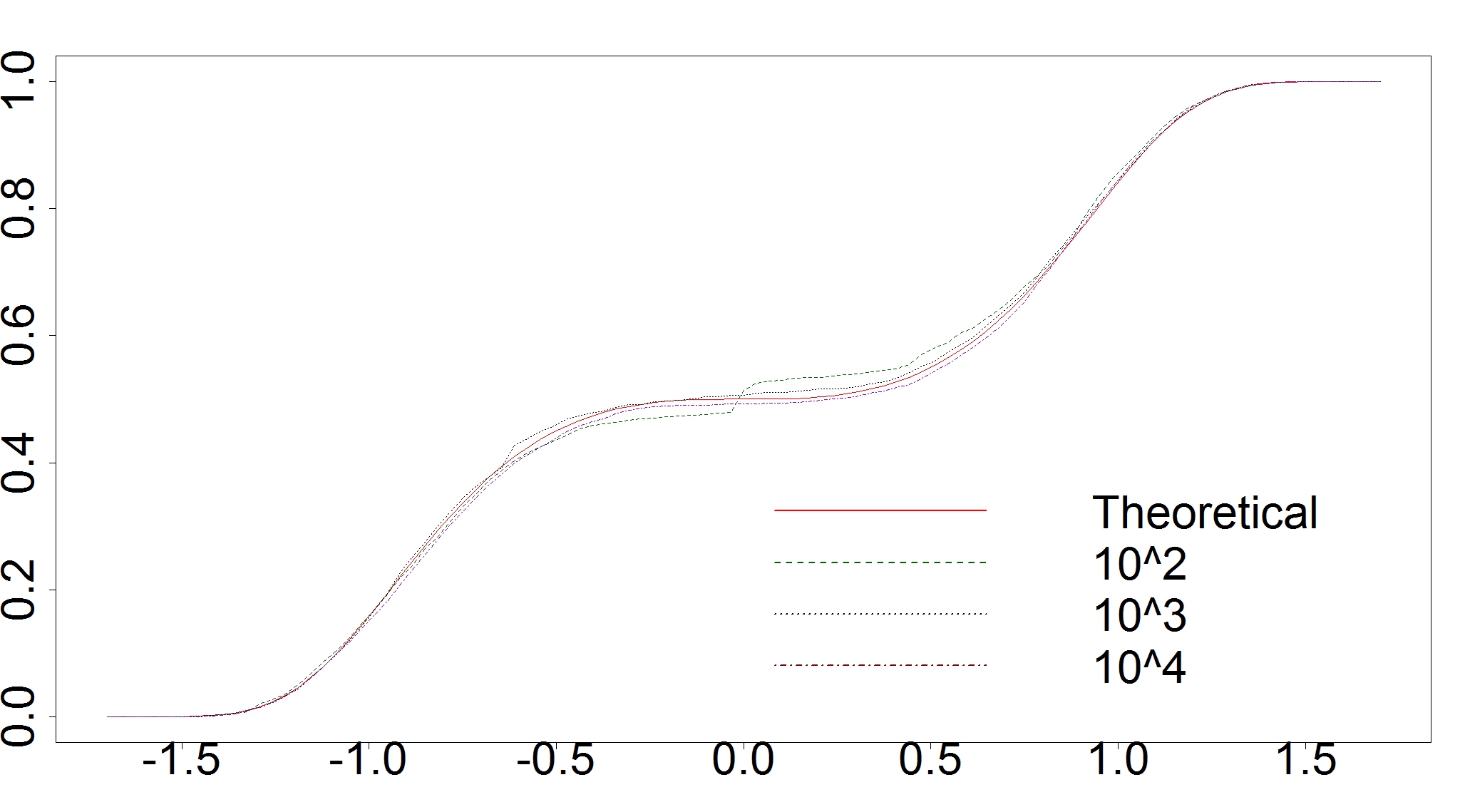}
		\caption{Quantile estimator}
		\label{subfig:quantile_ns_1}
	\end{subfigure}%
	\begin{subfigure}{.5\textwidth}
		\centering
		\includegraphics[width=\linewidth]{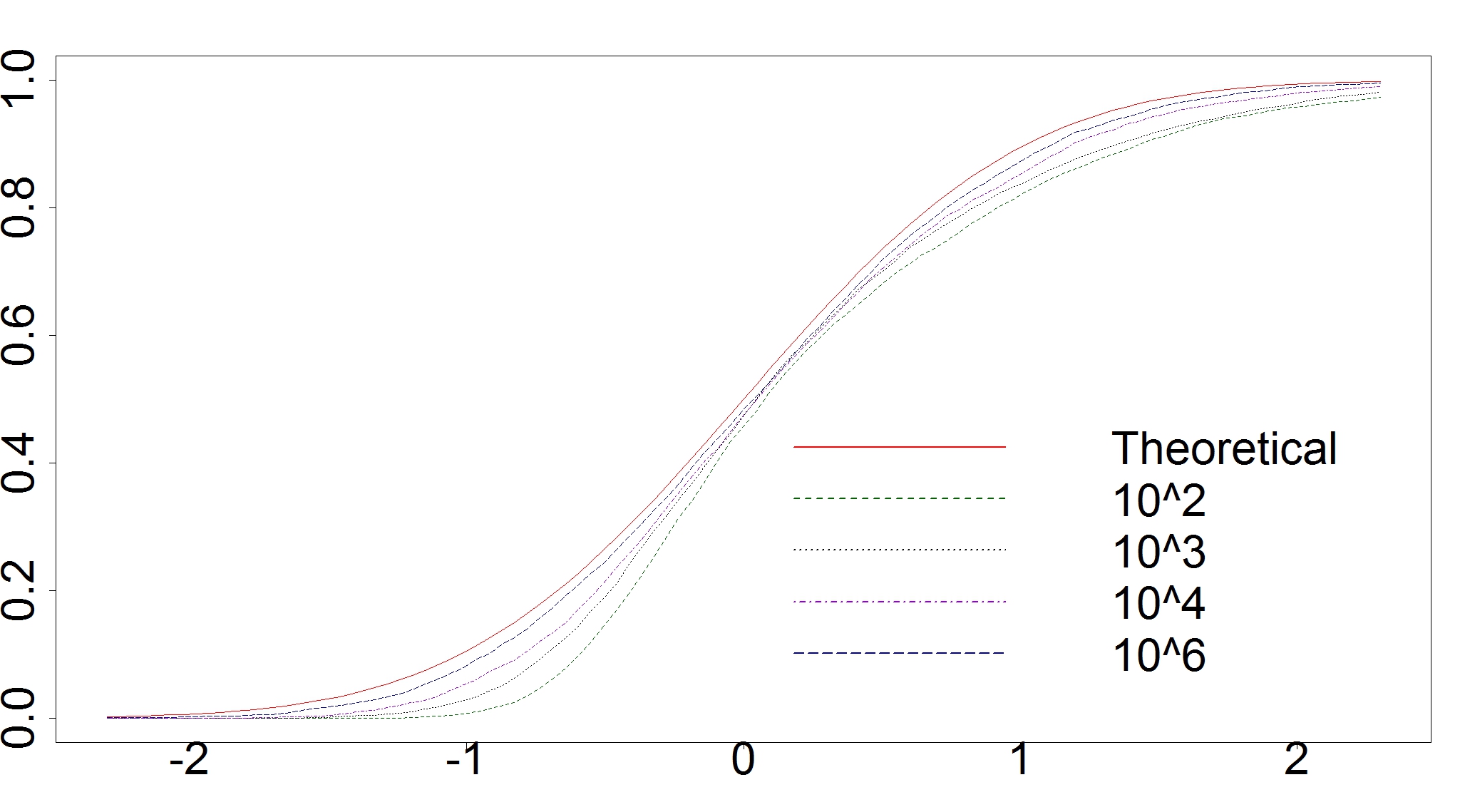}
		\caption{Expected Shortfall estimator}
		\label{subfig:expectedshort_ns_1}
	\end{subfigure}
	\captionsetup{font={small}, labelfont={bf}, width=.95\linewidth}
	\caption[ Example ~\ref{subsec:root_order_two}; Asymptotics of $\big(\widehat{q}_{n,\nicefrac{1}{2}},\widehat{es}_{n,\nicefrac{1}{2}}\big)$ for $m=5\cdot 10^3$]{The left picture (a) shows the limit (red solid) and the estimated distribution function of $n^{\nicefrac{1}{6}}\big(\widehat{q}_{n,\nicefrac{1}{2}}-1\big)$ for Section \ref{subsec:root_order_two} with $n=10^2$ (green dashed), $n=10^3$ (black dotted) and $n=10^4$ (purple dot-dashed), while the right picture (b) shows the estimated and the limit (red solid) distribution function of $\sqrt{n}\big(\widehat{es}_{n,\nicefrac{1}{2}}-\nicefrac{1}{4}\big)$ for $n=10^2$ (green dashed), $n=10^3$ (black dotted), $n=10^4$ (purple dot-dashed) and in addition $n=10^6$ (dark blue long-dashed).} 
	\label{fig:root_order_two_cdf_1}
\end{figure}

In Figure~\ref{fig:root_order_two_comm_dens_1} we used $n=10^6$ and increased the number of iterations to $m=5\cdot 10^4$ in order to nonparametrically estimate the joint density function of $\big(n^{\nicefrac{1}{6}}\big(\widehat{q}_{n,\nicefrac{1}{2}}-1\big),\sqrt{n}\big(\widehat{es}_{n,\nicefrac{1}{2}}-\nicefrac{1}{4}\big)\big)$, using the \textbf{R}-package \texttt{ks}, and compared this estimate to the density of the asymptotic distribution. 

\begin{figure}
	\centering
	\begin{subfigure}{.5\textwidth}
		\centering
		\includegraphics[width=\linewidth]{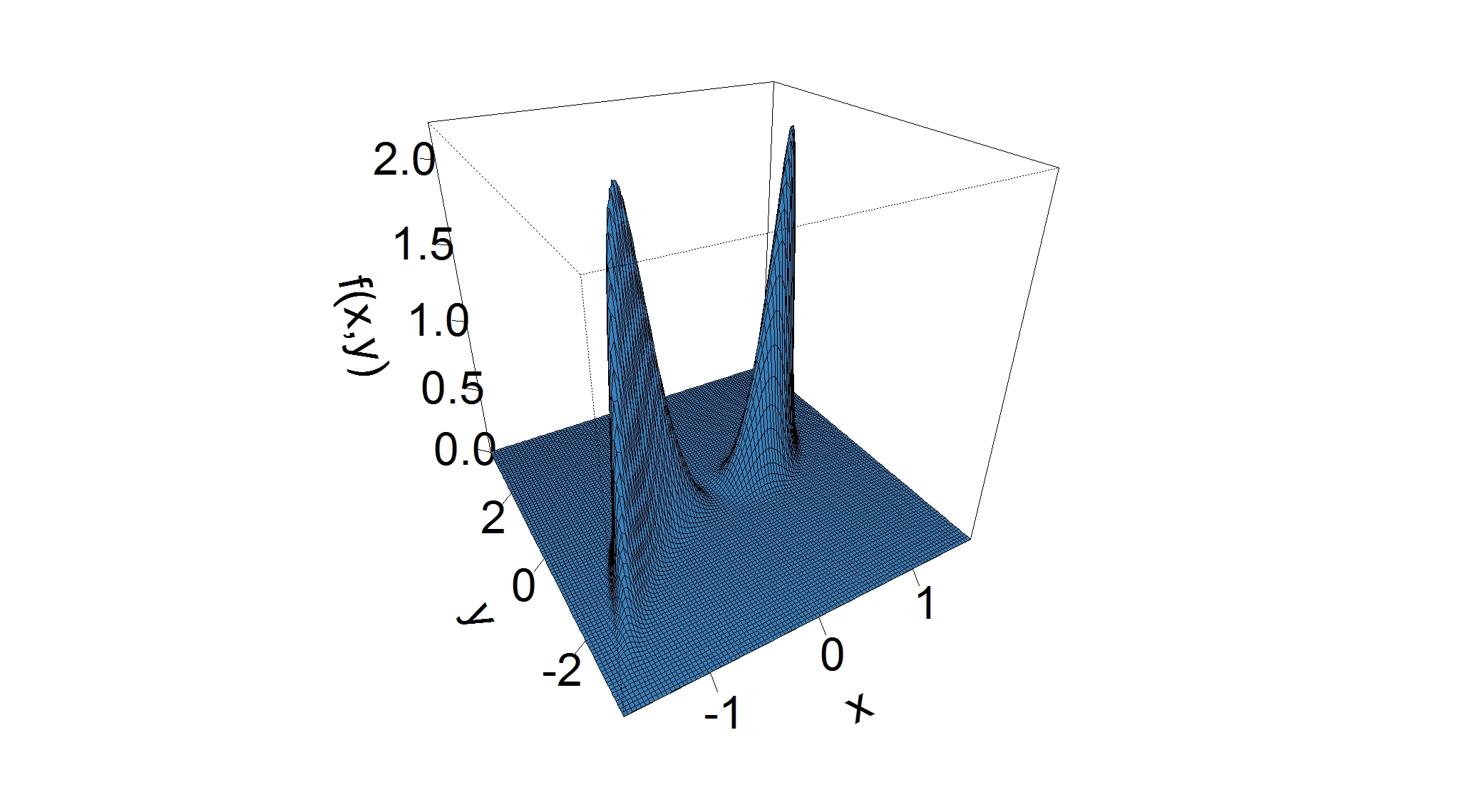}
		\captionsetup{font={small}, labelfont={bf}, width=.9\linewidth}
		\caption{Joint density of limit distribution}
		\label{subfig:theoretical_common_limit}
	\end{subfigure}%
	\begin{subfigure}{.5\textwidth}
		\centering
		\includegraphics[width=\linewidth]{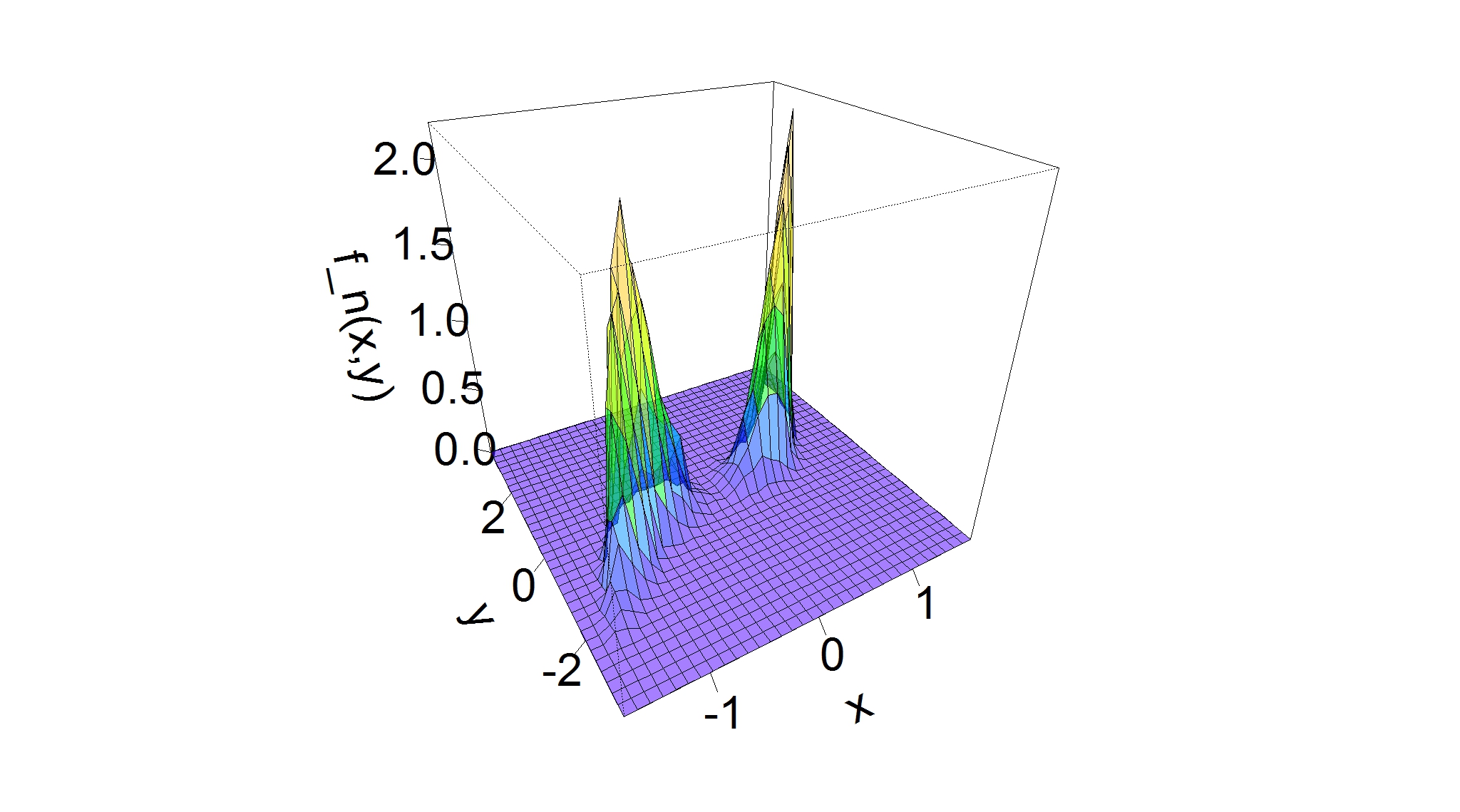}
		\captionsetup{font={small}, labelfont={bf}, width=.9\linewidth}
		\caption{Estimated joint density}
		\label{subfig:est_common_dens}
	\end{subfigure}
	\begin{subfigure}{.7\textwidth}
		\centering
		\includegraphics[width=\linewidth]{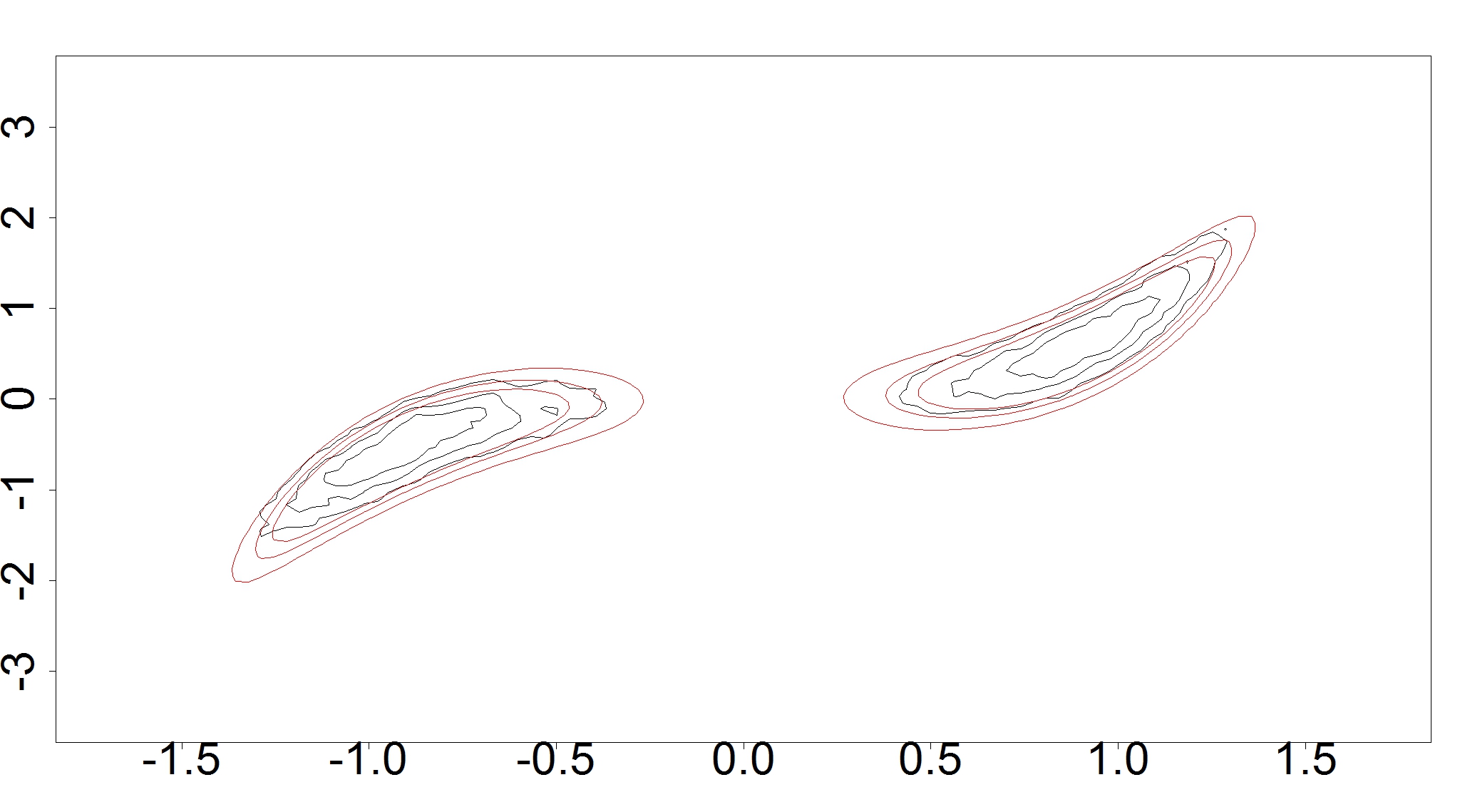}
		\captionsetup{font={small}, labelfont={bf}, width=.9\linewidth}
		\caption{Contour plots}
	\end{subfigure}
	\captionsetup{font={small}, labelfont={bf}, width=.9\linewidth}
	\caption[ Example~\ref{subsec:root_order_two}; Common asymptotics of $\big(\widehat{q}_{n,\nicefrac{1}{2}},\widehat{es}_{n,\nicefrac{1}{2}}\big)$ for $m=5\cdot 10^4$]{The above two images show (a) the limit joint density function and (b) the estimated joint density function of $\big(n^{\nicefrac{1}{6}}\big(\widehat{q}_{n,\nicefrac{1}{2}}-1\big),\sqrt{n}\big(\widehat{es}_{n,\nicefrac{1}{2}}-\nicefrac{1}{4}\big)\big)$. The image (c) shows the contour lines ($75\%,\, 50\%,\,25\%$ from outer to inner lines) of the theoretical (red) and the estimated (black) density in the above example. The shape of the theoretical distribution is captured well.}
	\label{fig:root_order_two_comm_dens_1}
\end{figure}

\renewcommand{\arraystretch}{1}


\section{Conclusions and discussion}\label{sec:discuss}

We show that the assumption of having a positive density at the $\alpha$-quantile, required for the quantile estimate to be asymptotically normal at $\sqrt{n}$-rate, are not required for asymptotic normality of the expected shortfall. 

The asymptotic variance of the ES can be estimated by forming a sample-counterpart expression. Alternatively, one may use the bootstrap. For the quantile in non-standard situations, \citet{knightboot} shows that the simple $n$-out-of-$n$ bootstrap is not consistent, but subsampling works. For the marginal asymptotic distribution of the expected shortfall, however, additional simulations indicate that the $n$-out-of-$n$ bootstrap is consistent, even without regularity of the density at the quantile. 

In this paper we only considered i.i.d.~data. Quantile and expectile estimation is often applied to financial time series, and therefore extensions of the results to dependent data would be useful. These should be possible but the details, in particular general M-estimation theory based on dependent data by using the argmax-continuity theorem, still need to be developed.  

Finally, the analysis of the expected shortfall as a process in the level $\alpha$ would be of some interest, in particular to study general spectral risk measures when not assuming a finitely-supported spectral measure.

\section*{Acknowledgements}

Tobias Zwingmann acknowledges financial support from of the Cusanuswerk for providing a dissertation scholarship.

\section{Proofs}\label{sec:proofs}
{\small 
	\subsection{Proofs of Propositions \ref{lem:connect_to_emp_vers}, \ref{thm:emp_vers_cons} and \ref{prop:theassumptionA}}
	
	We may write (\ref{eq:s_rep_1}) equivalently as
	\begin{align}\label{eq:s_rep_2}
	\begin{split}
	S(x_1, x_2; y) &=\big(1+\alpha^{-1}G(x_2)\big)\, \big(1_{ y \leq x_1} - \alpha \big) (x_1-y) + G(x_2) \, \big(x_2 -y\big) - \mathcal{G}(x_2) 
	\end{split}
	\end{align}

	\begin{proof}[{\sl Proof of Proposition \ref{lem:connect_to_emp_vers}}]
		Define the functions \begin{align*}
		\rho_{\alpha}(x_1;y) &= \big(1_{y\leq x_1}-\alpha\big)(x_1-y),\qquad 
		g(x_2)=\big(1+\alpha^{-1}\,G(x_2)\big) ,\\
		h(x_2;y)&= G(x_2)(x_2-y)-\mathcal{G}(x_2)
		\end{align*}
		so that 
		\[S(x_1,x_2;y)= g(x_2)\rho_{\alpha}(x_1;y)+h(x_2;y),\] 
		see (\ref{eq:s_rep_2}), hence 
		\begin{equation*}
		(\widehat{q}_{n,\alpha}, \widehat{es}_{n,\alpha}) \in \argmin_{(x_1,x_2)\in\R^2} \Big(\, g(x_2) \sum_{i=1}^{n}\rho_{\alpha}(x_1;X_i) + \sum_{i=1}^{n}h(x_2;X_i)\, \Big)
		\end{equation*}
		holds. The minimal value equals
		\begin{equation}\label{eq:minimizepartial}
		\min_{x_2\in\R}\, \Big( \,g(x_2) \Big(\min_{x_1\in\R} \sum_{i=1}^{n}\rho_{\alpha}(x_1;X_i)\Big) + \sum_{i=1}^{n}h(x_2;X_i)\, \Big),
		\end{equation}
		so that the minimizer in the first coordinate does not depend on the choice of $x_2$ and it follows that \begin{equation*}
		\widehat{q}_{n,\alpha} \in \argmin_{x_1\in\R} \sum_{i=1}^{n} \rho_{\alpha}(x_1;X_i) = [X_{\lceil n\,\alpha \rceil:n}, X_{\lceil n\,\alpha \rceil+1:n}),
		\end{equation*} 
		which includes the empirical quantile $\overline{q}_{n,\alpha}$.

		From (\ref{eq:minimizepartial}), $\widehat{es}_{n,\alpha}$ minimizes 
		\begin{equation*}
		x_2\mapsto \sum_{i=1}^{n} S\big(\widehat{q}_{n,\alpha},x_2;X_i\big) = \sum_{i=1}^{n}\, \big( g(x_2)\rho_{\alpha}(\widehat{q}_{n,\alpha};X_i)+h(x_2;X_i)\big).
		\end{equation*} 
		The partial derivatives of the functions $g, h$ are given by \begin{align*}
		\partial_{x_2} g(x_2) &= \alpha^{-1}G'(x_2) \quad\text{and} \quad 
		\partial_{x_2} h(x_2;y) = G'(x_2)(x_2-y).
		\end{align*}
		Thus \begin{align*}
		\partial_{x_2} \, \sum_{i=^1}^{n} S\big(\widehat{q}_{n,\alpha},x_2;X_i\big) &= G'(x_2) \,\sum_{i=1}^{n}\big(\alpha^{-1}\rho_{\alpha}(\widehat{q}_{n,\alpha};X_i)+x_2-X_i\big)\\
		&= G'(x_2) \,\sum_{i=1}^{n}\big(x_2-\widehat{q}_{n,\alpha} + \alpha^{-1}1_{X_i\leq \widehat{q}_{n,\alpha}}(\widehat{q}_{n,\alpha}-X_i)  \big).
		\end{align*}
		As $G'(x_2)>0$ by assumption, setting the above derivative equal to zero is equivalent to \begin{align*}
		0&=n\,x_2- n\, \widehat{q}_{n,\alpha} + \alpha^{-1}\sum_{i=1}^{n} 1_{X_i\leq \widehat{q}_{n,\alpha}}(\widehat{q}_{n,\alpha}-X_i).
		\end{align*} 
		By multiplying this with $1/n$ and reorganising the resulting equation the claim follows. 	
		
		For the final estimate, we observe that by the above,
		\begin{equation*}
		\Big|\widehat{es}_{n,\alpha}-\alpha^{-1}\mathbb{E}_n\big[Y \,1_{Y\leq \widehat{q}_{n,\alpha}}\big]\Big| = \alpha^{-1}\widehat{q}_{n,\alpha}\big|\alpha-\mathbb{E}_n[1_{Y\leq \widehat{q}_{n,\alpha}}]\big|.
		\end{equation*}
		So it remains to discuss $\big|\alpha-\mathbb{E}_n[1_{Y\leq \widehat{q}_{n,\alpha}}]\big|$: As $\widehat{q}_{n,\alpha} \in [X_{\lceil n\,\alpha\rceil:n},X_{\lceil n\,\alpha\rceil+1:n})$ we know that $\mathbb{E}_n[1_{Y\leq \widehat{q}_{n,\alpha}}] = \lceil n\, \alpha\rceil /n$ and thus we obtain \begin{equation*}
		\Big|\alpha - \mathbb{E}_n[1_{Y\leq \widehat{q}_{n,\alpha}}]\Big| = \frac{1}{n}\Big|n\,\alpha - \lceil n\, \alpha\rceil\Big|\leq \frac{1}{n}, 
		\end{equation*}
		which implies the assertion. 
	\end{proof}
	
	\begin{proof}[{\sl Proof of Proposition \ref{thm:emp_vers_cons}}]
		By the law of large numbers and the definition (\ref{eq:lowertailes}) of $es_{\alpha}$, 
		\begin{equation*}
		\Big|\alpha^{-1} \mathbb{E}_n\big[Y \,1_{Y\leq q_{\alpha}}\big]  - es_{\alpha}\Big| = o_{\mathcal{P}}(1).
		\end{equation*}
		
		This implies that
		\begin{align*}
		\Big|\alpha^{-1}\mathbb{E}_n[Y\,1_{Y\leq q_n}] - es_{\alpha}\Big| 
		&= \alpha^{-1}\Big| \mathbb{E}_n[Y\big(1_{Y\leq q_n} - 1_{Y\leq q_{\alpha}}\big)] \Big| + o_{\mathcal{P}}(1),
		\end{align*}
		and it remains to show that 
		\begin{equation}\label{eq:helpsecondlem}
		\mathbb{E}_n[Y\big(1_{Y\leq q_n} - 1_{Y\leq q_{\alpha}}\big)] \Big| = o_{\mathcal{P}}(1).
		\end{equation}
		Recall that since $\alpha\in (0,1)$ the quantile $q_{\alpha}$ is finite, hence $|q_{\alpha}|+1\leq c<\infty$.  Now let $\eta>0$ and choose $1\geq\delta>0$ such that $F(q_{\alpha}+\delta)-F(q_{\alpha}-\delta)\leq \alpha \,\eta/(2\,c)$ -- this is possible since $F$ is continuous in $q_{\alpha}$. On the set $\{|q_n-q_{\alpha}| \leq \delta\}$ the integral above is smaller than (or equal to) \begin{equation*}
		\max \{|q_\alpha-\delta|,|q_{\alpha}+\delta| \}\, \mathbb{E}_n \big[1_{Y \leq q_{\alpha}+\delta}-1_{Y\leq q_{\alpha}-\delta}\big] \leq c \,\mathbb{E}_n \big[1_{Y \leq q_{\alpha}+\delta}-1_{Y\leq q_{\alpha}-\delta}\big].
		\end{equation*}
		Next note that $\mathbb{E}_n \big[1_{Y \leq q_{\alpha}+\delta}-1_{Y\leq q_{\alpha}-\delta}\big]$ converges in probability to $\mathbb{E}(1_{Y \leq q_{\alpha}+\delta}-1_{Y\leq q_{\alpha}-\delta})$. Thus it follows that \begin{align*}
		&\mathcal{P}\Big(\alpha^{-1}\Big|  \mathbb{E}_n[Y\big(1_{Y\leq q_n} - 1_{Y\leq q_{\alpha}}\big)] \Big| \geq \eta \Big) \\[5pt]
		&\leq \mathcal{P}\big(|q_n-q_{\alpha}| \geq \delta \big) +\mathcal{P}\big( \mathbb{E}_n \big[1_{Y \leq q_{\alpha}+\delta}-1_{Y\leq q_{\alpha}-\delta}\big] \geq \alpha \,\eta / c  \big)\\[5pt]
		&\leq \mathcal{P}\big(|q_n-q_{\alpha}| \geq \delta \big) + \mathcal{P}\big( \big|(\mathbb{E}_n-\mathbb{E}) \big[1_{Y \leq q_{\alpha}+\delta}-1_{Y\leq q_{\alpha}-\delta}\big]\big| \geq \alpha \,\eta / (2\,c)  \big).
		\end{align*}
		The last two probabilities can be made small by choosing $n$ big enough ($|q_n-q_{\alpha}|=o_{\mathcal{P}}(1)$ and $\big|(\mathbb{E}_n-\mathbb{E}) \big[1_{Y \leq q_{\alpha}+\delta}-1_{Y\leq q_{\alpha}-\delta}\big]\big| =o_{\mathcal{P}}(1)$).
		
		For the statement concerning $\widetilde{es}_{n, \alpha}$, as in the proof of Proposition \ref{lem:connect_to_emp_vers} we obtain the generalization of (\ref{eq:esscoringexpl}), 
		\begin{equation*}
		\widetilde{es}_{n,\alpha} = \alpha^{-1}q_{n}\big(\alpha-\mathbb{E}_n \big[1_{Y\leq q_n}\big]\big) +\alpha^{-1} \mathbb{E}_n\big[Y\,1_{Y\leq q_n}\big].
		\end{equation*} 	
		
		Since from the first part of the proof, the last term above converges to $es_{\alpha}$ in probability, it remains to show $|\alpha^{-1}q_{n}\big(\alpha-\mathbb{E}_n \big[1_{Y\leq q_n}\big]\big)| = o_{\mathcal{P}}(1)$. For this it suffices to show $|\alpha-\mathbb{E}_n \big[1_{Y\leq q_n}\big]| = o_{\mathcal{P}}(1)$, as $\alpha^{-1}\,q_n$ is tight by assumption. The argument for this part is same as for (\ref{eq:helpsecondlem}). This concludes the proof of the proposition. 
	\end{proof}
	
	\begin{proof}[{\sl Proof of Proposition \ref{prop:theassumptionA}}]
		a.~The classification of $\psi_\alpha$ is shown in \citet[§~4]{smirnov1952limit}. Uniqueness of $(a_n)$ up to asymptotic equivalence follows from the convergence of types theorem and the distributional convergence of $a_n(\widehat{q}_{n,\alpha}-q_{\alpha})$ to a non-degenerate limit distribution under Assumption [A], see \citet{Knight2002LimDistr} or the proof of Theorem \ref{thm:main_thm_es}.\\ 
		b. If $(X_{\lceil n\,\alpha \rceil:n}-q_{\alpha})=o_{\mathcal{P}}(1)$, then one can find a sequence $a_n \to \infty$ for which $a_n(X_{\lceil n\,\alpha \rceil:n}-q_{\alpha}) = o_{\mathcal{P}}(1)$ is still true. 
		By Theorem~4, \citet{smirnov1952limit}, this holds if and only if \begin{equation}\label{eq:mot_ass_a}
		\frac{F(q_{\alpha}+t/a_n)-\lambda_{n,\alpha}}{\tau_{n,\alpha}} \to u(t).
		\end{equation}
		Here, $u:\R\to\overline{\R}$ is a non-decreasing function uniquely determined by \begin{equation*}
		1_{[0, \infty)}(t) = \frac{1}{\sqrt{2\pi}}\,\int_{-\infty}^{u(t)}\exp(-x^2/2)\,\dif x;
		\end{equation*} 
		further \begin{align*}
		\lambda_{n,\alpha} = \frac{\lceil n\,\alpha \rceil}{n+1},\quad \iota_{n,\alpha} = \frac{n-\lceil n\,\alpha \rceil + 1}{n+1}\quad \text{and}\quad \tau_{n,\alpha} = \sqrt{\frac{\lambda_{n,\alpha} \iota_{n,\alpha}}{n+1}}.
		\end{align*}
		With these definitions note that \begin{align*}
		\lambda_{n,\alpha} \to \alpha \quad \text{and}\quad \iota_{n,\alpha} \to 1-\alpha
		\end{align*}
		holds. Thus the convergence in (\ref{eq:mot_ass_a}) is equivalent to \begin{align*}
		\frac{\sqrt{n+1}\big(F(q_{\alpha}+t/a_n)-\alpha\big)}{\sqrt{\alpha(1-\alpha)}} \to u(t),
		\end{align*}
		which then yields the convergence assumed in [A] with $\psi_{\alpha}(t) =\sqrt{\alpha(1-\alpha)}\,u(t)$ and $a_n$ as chosen above.

	\end{proof}
	
	\subsection{General auxillary results}\label{sec:generalhelpres}

	The rates of convergence will be proved using the next theorem, which is a generalization of Theorem~5.52, \citet{vdv2000asympstat}, and similar to his  Theorem 5.23. We will provide a proof for convenience. Assume that $(\Theta_0, d_0), (\Theta_1, d_1)$ are metric spaces and that for all $\eta\in \Theta_0$, $\vartheta\in\Theta_1$, the map $y\mapsto m_{\eta, \vartheta}(y)$ is measurable.  To unify notation, we will use $\mathbb{E}_n$ and $Y$ in the formulation of the theorem, but note that $Y$ here could also have a more general form (not needing a finite first moment or $Y$ to be real).
	
	\begin{theorem}\label{thm:nuisance_rate_of_conv}
		Assume that for fixed $C$ and $\alpha>\beta$, every $n\in\N$ and all sufficiently small $\varepsilon,\delta>0$ it holds that 
		\begin{align}\label{eq:convrategen1}
		\inf_{d_0(\eta, \eta_0)\leq \varepsilon} \inf_{d_1(\vartheta, \vartheta_0)\geq \delta} \mathbb{E}\big[m_{\eta,\vartheta}(Y) - m_{\eta, \vartheta_0}(Y)\big] \geq C\delta^{\alpha}
		\end{align}
		and 
		\begin{equation}\label{eq:convrategen2}
		\mathbb{E}^*\Big[\sup_{d_0(\eta,\eta_0)\leq \varepsilon}\sup_{d_1(\vartheta, \vartheta_0)\leq \delta} \big|\sqrt{n}(\mathbb{E}_n-\mathbb{E})\big(m_{\eta, \vartheta}(Y) - m_{\eta, \vartheta_0}(Y)\big)\big|\Big] \leq C\delta^{\beta}. 
		\end{equation}
		Additionally suppose that $\eta_n$ converges to $\eta_0$ in (outer) probability and $\widehat{\vartheta}_n$ converges to $\vartheta_0$ in (outer) probability and fulfils \begin{equation*}
		\mathbb{E}_n\big[m_{\eta_n,\widehat{\vartheta}_n}(Y)\big]\leq  \mathbb{E}_n\big[m_{\eta_n, \vartheta_0}(Y)\big] + O_{\mathcal{P}}(n^{\nicefrac{\alpha}{(2(\beta - \alpha))}}).
		\end{equation*}
		Then $n^{\nicefrac{1}{(2(\alpha-\beta))}}d_1(\widehat{\vartheta}_n,\vartheta_0) = O_{\mathcal{P}}^*(1)$. 
	\end{theorem}
	
	For convenience, a proof of Theorem \ref{thm:nuisance_rate_of_conv} is provided in the technical supplement. 
	
	\medskip

	The next result is essential for obtaining the joint asymptotic distribution by use of the argmax-theorem when having different rates for the processes to be optimized.  
	
	\begin{lemma}\label{lem:diff_between_minimizer}
		Let $M_n$ and $M_n'$ be real valued processes, where $M_n'$ admits the representation \begin{align}\label{eq:rep_Mn_strich}
		M_n'(u_2) = N_n(u_2) + R_n
		\end{align}
		for $u_2\in\R^k$, where $R_n$ is a sequence of random variables not depending on $u_2$. Assume that
		\begin{equation}\label{eq:argmaxunifapprox}
		\sup_{u_2\in K} \big|M_n(u_2)-M_n'(u_2)\big| = o_{\mathcal{P}}(1)
		\end{equation}
		and that $N_n \Rightarrow N$ holds in $\ell^{\infty}(K_2)$ for every compact set $K_2\subset \R^k$ and some process $N$. Choose $(e_n,u_n)$ ($\in\R^{2k}$) as minimizer of $(M_n,M_n')$ and assume in addition that $e_n=O_{\mathcal{P}}(1)$ and $u_n \Rightarrow u_0$ (as variables in $\R^k$), where $u_0$ is the unique minimizer of $N$ (assuming all of these variables exist). 	Then 
		\[ e_n = u_n + o_{\mathcal{P}}(1).\]
	\end{lemma}
	
	\begin{remark*}
		The sequence of processes $(M_n)$ need not converge and hence the argmax-continuity theorem cannot be applied directly to the minimizers $(e_n)$. The approximating processes $(M_n')$ converge apart from a sequence of random variables $(R_n)$ not depending on $u_2$.  
	\end{remark*}

	\begin{proof}[Proof of Lemma~\ref{lem:diff_between_minimizer}.]
		Let
		\begin{align*}
		\overline{M}_n(u_2, u_2') &= M_n(u_2) + M_n'(u_2'), \qquad 	\overline{M}_n'(u_2, u_2') = M_n'(u_2) + M_n'(u_2'),\\
		\overline{N}_n(u_2,u_2') &= N_n(u_2)+N_n(u_2'), \qquad 
		\overline{N}(u_2,u_2') = N(u_2)+N(u_2').
		\end{align*} 
		For $B\subset \R^k$ set
		\begin{equation*}
		N_n(B)=\inf_{u_2\in B} N_n(u_2)
		\end{equation*}
		and similarly for $N(B), M_n(B), M_n'(B)$ as well as $\overline{M}_n(C), \overline{M}_n'(C),  \overline{N}_n(C), \overline{N}(C)$ for $C\subset \R^{2k}$.
		
		We shall show $(e_n,u_n)\Rightarrow (u_0,u_0)$, so that from the continuous mapping theorem we deduce that $(e_n-u_n)$ converges to $0$ weakly and thus in probability. 
		
		For the weak convergence of $(e_n,u_n)$ we utilize the Portmanteau Theorem. Let $F\subset \R^{2k}$ be closed and let $\varepsilon>0$. Since $e_n=O_P(1)$ and $u_n=O_P(1)$ by assumption we can find a compact set $K\subset\R^2$ for which $P((e_n,u_n) \notin K)\leq \varepsilon$ and $P((u_0, u_0) \notin K)\leq \varepsilon$. 
		From (\ref{eq:argmaxunifapprox}) and the representation of $M_n'$ we have that
		\[ \overline{M}_n(F \cap K) = \overline{M}_n'(F \cap K) + o_{\mathcal{P}}(1) = \overline{N}_n(F \cap K) + o_{\mathcal{P}}(1) + 2\, R_n,\]
		and similarly for $\overline{M}_n(K)$. 
		Now if $(e_n,u_n) \in F\cap K$, then $\overline{M}_n(F\cap K)\leq \overline{M}_n(K)$ holds, and by the above this implies $\overline{N}_n(F\cap K)\leq \overline{N}_n(K)+o_P(1)$, thus
		\begin{equation}\label{eq:par_to_value}
		P((e_n,u_n) \in F\cap K) \leq P\Big(\overline{N}_n(F\cap K)\leq \overline{N}_n(K)+o_P(1)\Big).
		\end{equation}
		The process $N_n$ is asymptotically tight by assumption, hence $(N_n,N_n)$ is asymptotically tight by Lemma~1.4.3, \citet{vdvwellner1996weak}. The convergence of the finite dimensional distributions is fulfilled as $N_n\Rightarrow N$ and thus Theorem~1.5.4 of \citet{vdvwellner1996weak} yields $(N_n,N_n)\Rightarrow (N,N)$.
		Therefore, by the continuous mapping theorem the weak convergence $N_n(u_2) + N_n(u_2')\Rightarrow N(u_2)+N(u_2')$ in $\ell^{\infty}(K_2)$ for any $K_2\subset\R^{2k}$ compact follows. Hence -- again due to the continuous mapping theorem -- the convergence $\big(\overline{N}_n(F\cap K), \overline{N}_n(K)\big) \Rightarrow \big(\overline{N}(F\cap K), \overline{N}(K)\big)$ holds. Then Slutsky's lemma and the portmanteau lemma imply 
		\begin{equation}\label{eq:overline_N}
		\mathcal{P} \Big(\overline{N}_n(F\cap K)\leq \overline{N}_n(K)+o_P(1)\Big) \leq \mathcal{P}\Big( \overline{N}(F\cap K)\leq \overline{N}(K)\Big) + o(1).
		\end{equation}
		Since $(u_0,u_0)$ is the unique minimizer of $\overline{N}$ by assumption, on the event $\{(u_0,u_0)\in F^c\}$ the inequality $\overline{N}(u_0,u_0) < \overline{N}(F\cap K)$ is fulfilled. If we additionally are on the event $ \{ \overline{N}(F\cap K) \leq \overline{N}(K)\}$ we can deduce that $\overline{N}(u_0,u_0) < \overline{N}(K)$ must hold, hence $(u_0, u_0)\notin K$. This means 
		\begin{align}
		&\mathcal{P}\Big( \overline{N}(F\cap K) \leq \overline{N}(K) \Big) \leq \mathcal{P}\big((u_0,u_0)\notin K\big) + \mathcal{P}\big((u_0,u_0)\in F\big)\label{eq:set_disj_part}.
		\end{align} 
		Combining (\ref{eq:par_to_value}), (\ref{eq:overline_N}) and (\ref{eq:set_disj_part}) gives  
		\begin{equation}\label{eq:ineq_portmanteau}
		\limsup_{n\to\infty} \mathcal{P}((e_n,u_n)\in F\cap K) \leq \mathcal{P}((u_0,u_0)\in F) + \mathcal{P}((u_0,u_0)\notin K).
		\end{equation}

		Now by the choice of $K$ we have $P((u_0,u_0)\notin K)\leq \varepsilon$ and 
		\begin{align*}
		\limsup_n \mathcal{P}((e_n,u_n)\in F\cap K) &\geq \limsup_n \mathcal{P}((e_n,u_n)\in F)-\sup_n \mathcal{P}((e_n,u_n)\notin K) \\
		&\geq \limsup_n \mathcal{P}((e_n,u_n)\in F)-\varepsilon,
		\end{align*}
		so it follows that \begin{equation*}
		\limsup_n \mathcal{P}((e_n,u_n) \in F) \leq P((u_0,u_0) \in F) + 2\varepsilon. 
		\end{equation*}
		Since $\varepsilon$ was arbitrary the portmanteau lemma yields $(e_n,u_n)\Rightarrow (u_0,u_0)$. 
	\end{proof}


	\subsection{Proof of Theorem \ref{thm:main_thm_es}}
	%
	
	
	In this subsection we give the main steps of the proof of Theorem \ref{thm:main_thm_es}. Proofs of intermediate lemmas are either given in the following subsection, or deferred to the technical supplement.  \\
	
	{\bf Step 1} {\sl Increments of the scoring function}\\
	
	We start with a technical lemma on increments of the scoring function. 
	\begin{lemma}\label{lem:dif_s_reformulated}
		a.		We have that
		\begin{align}
		&S(x_1,x_2+y_2; z) - S(x_1,x_2;z) \notag\\
		&=   \big(G(x_2+y_2)- G(x_2) \big)\big( x_2 - x_1 +\alpha^{-1} 1_{ z \leq x_1}\big(x_1-z \big)\big) + \int_{0}^{y_2} G'(x_2+s)s \,\dif s\label{eq:diff_s_1}\\
		&=\big(G(x_2+y_2)- G(x_2) \big) \big( x_2 - x_1 +\alpha^{-1} 1_{ z \leq x_1}\big(x_1-z \big)\big) \notag\\
		&\quad +  \frac{1}{2}G'(x_2+y_2) y_2^2 - \frac{1}{2} \int_{0}^{y_2} G''(x_2+s)s^2 \,\dif s\label{eq:diff_s_2}.
		\end{align}
		b. Setting $\rho_{\alpha}(x_1;z) = \big(1_{z\leq x_1}-\alpha\big)(x_1-z)$ we have that 
		\begin{equation}\label{eq:incrementquant}
		\rho_{\alpha}(x_1+y_1; z) - \rho_{\alpha}(x_1;z) = y_1\big(1_{z \leq x_1}-\alpha\big) + \int_0^{y_1} \big(1_{z\leq x_1+s}-1_{z\leq x_1}\big) \, \dif s.
		\end{equation}
		
		c. Generally, we have that
		\begin{align}
		S&(x_1+y_1, x_2+y_2;z) - S(x_1,x_2;z)\notag\\
		&= \big(1+\alpha^{-1}G(x_2+y_2)\big) \Bigg(y_1\big(1_{z\leq x_1}-\alpha\big) + \int_0^{y_1}1_{z\leq x_1+s}-1_{z\leq x_1} \, \dif s\Bigg)\notag\\[5pt]
		&\quad + \big(G(x_2+y_2)-G(x_2)\big)\big( x_2 - x_1 +\alpha^{-1} 1_{ y \leq x_1}\big(x_1-y \big)\big) \notag\\[5pt]
		&\quad +\frac{1}{2}G'(x_2+y_2)y_2^2 - \frac{1}{2}\int_0^{y_2}G''(x_2+s)s^2\,\dif s\label{eq:difference_s}.
		\end{align}
		
	\end{lemma}
	
	The proof of Lemma \ref{lem:dif_s_reformulated} is relegated to the technical supplement. \\

	{\bf Step 2} {\sl $\sqrt{n}$-rate of convergence of $\widehat{es}_{n,\alpha}$ or, more generally, $\widetilde{es}_{n, \alpha}$.} \\
	
	The following lemma is proved by checking the assumptions of Theorem \ref{thm:nuisance_rate_of_conv}. 

	\begin{lemma}\label{cor:prep_es_tight}
		Assume $q_{n}$ to be a consistent estimator of $q_{\alpha}$ and [B] to hold, then the sequence $\sqrt{n}(\widetilde{es}_{n, \alpha}-es_{\alpha})$ is tight, where $\widetilde{es}_{n, \alpha}$ is the minimizer of the function \begin{equation*}
		x_2 \mapsto\sum_{i=1}^{n} S(q_{n}, x_2; X_i) = n\mathbb{E}_n[S(q_n,x_2;Y)].
		\end{equation*}
		In particular, if [A] and [B] hold, then $\sqrt{n}\big(\widehat{es}_{n,\alpha}-es_{\alpha}\big)$ is a tight sequence. 
	\end{lemma}
	
	{\bf Step 3} {\sl Convergence of processes to be minimized}\\
	
	Using (\ref{eq:difference_s}) with $x_1=q_{\alpha}$, $y_1=u_1/a_n$, $x_2=es_{\alpha}$ and $y_2= u_2/\sqrt{n}$, where $a_n>0$, we may write
	\begin{align}\label{processrepres}
	\begin{split}
	&\sum_{i=1}^{n} S\big(q_{\alpha}+u_1/a_n, es_{\alpha}+u_2/\sqrt{n}; X_i)-S(q_{\alpha}, es_{\alpha}; X_i)\\
	=\, & \, \big(1+\alpha^{-1}G(es_{\alpha}+u_2/\sqrt{n})\big)V_n(u_1) + U_n(u_2),
	\end{split}
	\end{align}
	where 
	\begin{align}
	V_n(u_1) & = \frac{u_1}{a_n}\sum_{i=1}^{n}(1_{X_i\leq q_{\alpha}}-\alpha)+ \frac{1}{a_n}\int_0^{u_1}\Big(\sum_{i=1}^{n}1_{X_i\leq q_{\alpha}+t/a_n}-1_{X_i\leq q_{\alpha}}\Big) \, \dif t\notag \\
	& = \sum_{i=1}^n \rho_{\alpha}\big(q_{\alpha}+u_1/a_n;X_i\big)-\rho_{\alpha}\big(q_{\alpha};X_i \big) \label{eq:provessvn}
	\end{align}
	and
	\begin{align}
	U_n(u_2) & = \sqrt{n}\big(G(es_{\alpha}+u_2/\sqrt{n})-G(es_{\alpha})\big) \frac{1}{\sqrt{n}}\sum_{i=1}^{n}\big(es_{\alpha}-q_{\alpha}+\alpha^{-1}1_{X_i\leq q_{\alpha}}(q_{\alpha}-X_i)\big) \notag\\
	&\quad +\frac{u_2^2}{2}G'(es_{\alpha}+u_2/\sqrt{n})- \frac{1}{2\sqrt{n}}\int_0^{u_2} G''(es_{\alpha}+t/\sqrt{n})t^2\,\dif t\\
	& = \sum_{i=1}^{n} S(q_{\alpha},es_{\alpha}+\nicefrac{u_2}{\sqrt{n}};X_i) - S(q_{\alpha},es_{\alpha};X_i). \label{eq:provessun}
	\end{align}
	Here we used (\ref{eq:incrementquant}) and (\ref{eq:diff_s_2}) and made a substitution in the integrals. 
	
	Under the Assumptions [A] and [B] the processes $U_n$ and the rescaled processes $\big(\frac{a_n}{\sqrt{n}}\big)\,V_n$ converge in distribution.

	\begin{lemma}\label{thm:conv_to_V_U}
		%
		If [A] holds, then 
		\begin{equation}\label{eq:processconvlem1}
		\frac{a_n}{\sqrt{n}}\,V_n(u_1) \Rightarrow u_1 W_1 + \int_0^{u_1} \psi_{\alpha}(t)\,\dif t =:V(u_1)
		\end{equation}
		in $\ell^{\infty}(K_1)$ for every compact set $K_1\subset \R$, where $W_1\sim\mathcal{N}(0,\alpha(1-\alpha))$. 
		
		If [B] holds we have the convergence 
		\begin{equation}\label{eq:processconvlem2}
		U_n(u_2)\Rightarrow G'(es_{\alpha})\, \big(\, u_2\, W_2 + u_2^2\,  /2\big) =:U(u_2)
		\end{equation} 
		in $\ell^{\infty}(K_2)$ for every compact set $K_2\subset\R$ and  	
		$ W_2\sim\mathcal{N}\big(0, \Var \big[1_{Y \leq q_\alpha}\, (q_\alpha - Y)/\alpha \big]\big)
		$. \\
		Moreover, if both [A] and [B] hold we have that 
		\begin{equation}\label{eq:processconvlem3}
		\big((a_n/\sqrt{n}) V_n, U_n\big) \Rightarrow (V,U)
		\end{equation}
		in $\ell^{\infty}(K)$ for compact $K\subset \R^2$, where $(W_1,W_2)$ in the definition of $(V,U)$ are jointly normal with covariance $(1-\alpha)\, (q_\alpha - es_\alpha)$. 
	\end{lemma}
	
	{\bf Step 4} {\sl Approximation of $\sqrt{n}(\widehat{es}_{n,\alpha}-es_{\alpha})$ by minimizer of $U_n(u_2)$.} \\
	
	\begin{lemma}\label{lem:approxminimizer}
		The processes $(U_n)$ and $U$ in Lemma~\ref{thm:conv_to_V_U} have unique minimizers $({u}_{2,n})$ and $u_{2}^0$, and ${u}_{2,n} \Rightarrow u_{2}^0$. Moreover, we have that
		\begin{equation}\label{eq:minim_op_one}
		\sqrt{n}(\widehat{es}_{n,\alpha}-es_{\alpha})  = {u}_{2,n}+o_{\mathcal{P}}(1).
		\end{equation}
	\end{lemma}

	{\bf Step 5} {\sl Application of the argmax continuity theorem.} \\
	
	From Lemma \ref{lem:approxminimizer}, we have that 
	\[ (a_n(\widehat{q}_{n,\alpha}-q_{\alpha}), \sqrt{n}(\widehat{es}_{n,\alpha}-es_{\alpha})\big) = \big(a_n(\widehat{q}_{n,\alpha}-q_{\alpha}), u_{2,n}\big) + o_{\mathcal{P}}(1).\]
	Now, $\big(a_n(\widehat{q}_{n,\alpha}-q_{\alpha}), u_{2,n}\big)$ is by construction a sequence of minimizers of the processes 
	$ \big(V_n(u_1) + U_n(u_2)\big)$, 
	but since the variables are separated, also of the processes
	\[	Z_n(u_1,u_2)\\
	=\big(1+\alpha^{-1}G(es_{\alpha})\big)\, \frac{a_n}{\sqrt{n}}\, V_n(u_1) + U_n(u_2),
	\]
	which, by Lemma \ref{thm:conv_to_V_U}, (\ref{eq:processconvlem3}), and the continuous mapping theorem, converge in $\ell^{\infty}(K)$ for compact $K$ to the process 
	\begin{equation*}
	Z(u_1,u_2)=(1+\alpha^{-1}G(es_{\alpha}))\, V(u_1) + U(u_2).
	\end{equation*}
	To conclude$\big(a_n(\widehat{q}_{n,\alpha}-q_{\alpha}), u_{2,n}\big) \Rightarrow (z_1,z_2)$, the minimizer of $Z$, we apply the argmax-continuity theorem, e.g.~Corollary~5.58, \citet{vdv2000asympstat}, and need to check the remaining assumptions. 
	
	The process $U$ apparently has a unique minimizer, and $\sqrt{n}(\widehat{es}_{n,\alpha}-es_{\alpha})$ is a tight sequence by Lemma \ref{cor:prep_es_tight}. 
	
	The process $V$ also has a unique minimum almost surely. Indeed, the form of the functions $\psi_{\alpha}(t)$ as given in Proposition \ref{prop:theassumptionA}, in particular that $\kappa_+, \kappa_-, \beta >0$, as well as the form of $V$ in (\ref{eq:processconvlem1}) imply that 
	$\lim_{u_1\to\pm\infty} V(u_1)=\infty$ and that for the closed interval for which $|V(u_1)|< \infty$ the derivative has at most one zero; if it has no zero the minimizer is on the boundary of this interval. 
	Moreover, in the proof of Lemma \ref{lem:approxminimizer} we already observed that $a_n(\widehat{q}_{n,\alpha}-q_{\alpha})$ is a tight sequence. This concludes the proof of Theorem \ref{thm:main_thm_es}.

	
	\subsubsection{Proofs of intermediate results}

	\begin{proof}[{\sl Proof of Lemma \ref{cor:prep_es_tight}}]
		
		The proof proceeds by checking the assumptions of Theorem~\ref{thm:nuisance_rate_of_conv} with $\alpha=2$, $\beta=1$, $d_0, d_1$ the Euclidean distance in $\R$ and the criterion function $m_{\eta, \vartheta}(y) = S(\eta, \vartheta; y)$. We give an outline here, full details are provided in the technical supplement. 
		
		Consistency has been taken care of in Theorem~\ref{thm:emp_vers_cons}.
		
		Concerning (\ref{eq:convrategen1}) we need to prove that 
		\begin{equation}\label{eq:helprate1tech2}
		\inf_{|\eta-q_{\alpha}|\leq \varepsilon} \inf_{\delta_0\geq |\vartheta-es_{\alpha}|\geq \delta} \mathbb{E}\left[S(\eta, \vartheta; Y)-S(\eta, es_{\alpha};Y)\right] \geq C\,\delta^2,
		\end{equation}
		$0 < \varepsilon \leq c_0 $, $0<\delta \leq \delta_0$ for some $C, \delta_0, c_0>0$.

		To this end, using (\ref{eq:diff_s_1}) in Lemma~\ref{lem:dif_s_reformulated}, we get by convexity of  $\eta\mapsto \mathbb{E}\big(\big(1_{Y\leq \eta}-\alpha\big)\big(\eta-Y\big)\big)$
		that 		
		\begin{align*}
		\inf_{|\eta-q_{\alpha}|\leq \varepsilon} \inf_{\delta_0\geq |\varepsilon_2|\geq \delta} \mathbb{E}\left[S(\eta, es_{\alpha}+\varepsilon_2;Y)-S(\eta, es_{\alpha};Y)\right]
		\geq \frac{\delta^2}{2}\Big(G'(es_{\alpha} \pm \delta) - \delta \,C/3\Big)
		\end{align*}
		for $C=\sup_{x_2\in [es_{\alpha}-\delta_0,es_{\alpha}+\delta_0 ]} |G''(x_2)|<\infty$. Since $\lim_{\delta \to 0} G'(es_{\alpha}\pm \delta) - \delta \,C/3 = G'(es_{\alpha}) >0$ holds by assumption on $G'$ we can find $C'>0$ and $\delta_0>0$ with $G'(es_{\alpha}\pm \delta) - \delta \,C/3 \geq C'$ for every $\delta \leq \delta_0$. This proves (\ref{eq:helprate1tech2}). 
		
		For (\ref{eq:convrategen2}) we require  
		\begin{equation}\label{eq:helprate1tech}
		\mathbb{E}\Big[ \sup_{ \substack{\scriptscriptstyle |\eta-q_{\alpha}|\leq \varepsilon\\\scriptscriptstyle |\vartheta-es_{\alpha}|\leq \delta}} \Big|\sqrt{n}(\mathbb{E}_n-\mathbb{E})\big( S(\eta, \vartheta;Y)-S(\eta, es_{\alpha};Y)\big)\Big| \Big] \leq C \,\delta, \qquad 0 < \delta \leq \delta_0,
		\end{equation}
		for some $\delta_0, C >0$. 	

		Proving (\ref{eq:diff_s_1}) reduces to showing that 
		\begin{equation}\label{eq:expect_sup_bound}
		\mathbb{E}\left[\sup_{|\eta-q_{\alpha}|\leq \varepsilon}\big|\sqrt{n}(\mathbb{E}_n-\mathbb{E})(1_{Y\leq \eta}(\eta-Y))\big|\right] \leq K
		\end{equation}
		for some constant $K$ not depending on $\delta$, which may be accomplished by using a maximal inequality involving the bracketing integral (Definition in Chapter~19.2, \citet{vdv2000asympstat}).  
	\end{proof}

	\begin{proof}[Proof of Lemma~\ref{thm:conv_to_V_U}.]
		
		Assume [A]. In fact, the convergence (\ref{eq:processconvlem1}) was shown in \citet{Knight2002LimDistr}. For convenience, we give a (different) proof here. 
		First we observe that
		\begin{equation}\label{eq:l_cont_of_rho}
		\big|\rho_{\alpha}(x_1;y) - \rho_{\alpha}(x'_1;y)\big|\leq (1+\alpha)\big|x_1-x_1'\big|.
		\end{equation}
		Indeed, if $x_1\leq x_1' <y$, then $\big|\rho_{\alpha}(x_1;y) - \rho_{\alpha}(x'_1;y)\big| = \alpha(x_1'-x_1)$ is satisfied. Else, if $x_1 <y \leq x_1'$, then \begin{align*}
		\big|\rho_{\alpha}(x_1;y) - \rho_{\alpha}(x'_1;y)\big|  &= \big|-\alpha\,x_1-x_1'+\alpha\,x_1'+y\big| \leq \alpha(x_1'-x_1) + \big( x_1'-y\big)\\
		&\leq \alpha(x_1'-x_1) + \big( x_1'-x_1\big) = (1+\alpha)(x_1'-x_1)
		\end{align*}
		is valid. In the last case where $y\leq x_1\leq x_1'$ it holds that $\big|\rho_{\alpha}(x_1;y) - \rho_{\alpha}(x'_1;y)\big| = (1-\alpha)(x_1'-x_1)$. All three cases together prove (\ref{eq:l_cont_of_rho}). 
		
		Using the Lipschitz continuity (\ref{eq:l_cont_of_rho}), from Lemma~19.31 in \citet{vdv2000asympstat} we obtain that 
		\begin{equation*}
		\mathbb{G}_n\Big[a_n\big(\rho_{\alpha}\big(q_{\alpha}+u_1/a_n;Y\big)-\rho_{\alpha}\big(q_{\alpha};Y \big) \big) - u_1(1_{Y\leq q_{\alpha}}-\alpha)\Big] = o_\mathcal{P}(1). 
		\end{equation*}
		
		Therefore, from the definition of $V_n(u_1)$ in (\ref{eq:provessvn}), 
		\begin{align}
		\frac{a_n}{\sqrt{n}}V_n(u_1)&= \sqrt{n}\mathbb{E}_n\big[a_n\big(\rho_{\alpha}(q_{\alpha}+u_1/a_n;Y)-\rho_{\alpha}(q_{\alpha};Y)\big)\big]\notag\\
		&= \frac{u_1}{\sqrt{n}}\sum_{i=1}^{n}\big(1_{X_i\leq q_{\alpha}}-\alpha\big) + \int_0^{u_1}\sqrt{n}\big(F(q_{\alpha}+s/a_n)-F(q_{\alpha})\big)\,\dif s + o_{\mathcal{P}}(1).\label{eq:expansionprocess1}
		\end{align}
		The first term converges by the central limit theorem to $u_1W_1$ for $W_1$ as stated. For the second, note that under Assumption [A] we also have
		\begin{equation*}
		\lim_{n\to\infty} \int_0^t \sqrt{n}\big(F(q_{\alpha}+s/a_n)-F(q_{\alpha})\big)\,\dif s = \int_0^t \psi_{\alpha}(s)\,\dif s.
		\end{equation*}
		Indeed, using the monotonicity of $\sqrt{n}\big(F(q_{\alpha}+s/a_n)-F(q_{\alpha})\big)$, this follow from the dominated convergence theorem if $|\psi_{\alpha}(t)| < \infty$, and the fact that $\{\psi_{\alpha} = \infty\}$ and $\{\psi_{\alpha} = -\infty\}$ are open intervals (see Proposition \ref{prop:theassumptionA}, a.) if $|\psi_{\alpha}(t)| = \infty$. 
		
		Now assume [B]. Below we show that for every compact set $K_2$ with $es_{\alpha}\in K_2$ there exists a function $m(y)$ such that for every $x_2,x_2'\in K_2$,
		\begin{equation}\label{eq:l_cont_of_rest}
		\big|S(q_{\alpha}, x_2';y)-S(q_{\alpha}, es_{\alpha};y)\big|\leq m(y)\big|x_2'-x_2\big|,
		\end{equation}
		where $m(y)$ fulfills $\mathbb{E}[m(Y)]<\infty$.
		
		To deduce (\ref{eq:processconvlem1})	we again apply Lemma~19.31, in \citet{vdv2000asympstat}. From (\ref{eq:l_cont_of_rest}) and since 
		\begin{equation*}
		\partial_{x_2} S(q_{\alpha}, x_2;y) \, (es_{\alpha}) = G'(es_{\alpha})\big(es_{\alpha}-q_{\alpha}+\alpha^{-1}1_{y\leq q_{\alpha}}(q_{\alpha}-y)\big).
		\end{equation*}
		we thus obtain that 
		\begin{equation*}
		\mathbb{G}_n\Big[\sqrt{n}\big(S\big(q_{\alpha}, es_{\alpha} + u_2/\sqrt{n};Y\big)-S(q_{\alpha}, es_{\alpha} ;Y)\big) - u_2\,G'(es_{\alpha})\big(es_{\alpha}-q_{\alpha}+\alpha^{-1}1_{Y\leq q_{\alpha}}(q_{\alpha}-Y)\big)\Big] 
		\end{equation*}
		converges to zero in probability in $\ell^{\infty}(K_2)$. Using (\ref{eq:incrementquant}) this implies
		\begin{align}
		&U_n(u_2)\notag\\
		&= \sqrt{n}\,\mathbb{E}_n\Big[\sqrt{n}\big(S(q_{\alpha}, es_{\alpha}+u_2/\sqrt{n};Y)-S(q_{\alpha}, es_{\alpha};Y)\big)\Big]\notag\\[5pt]
		&=u_2\sqrt{n}\mathbb{E}_n\big[G'(es_{\alpha})\big(es_{\alpha}-q_{\alpha}+\alpha^{-1}1_{Y\leq q_{\alpha}}(q_{\alpha}-Y)\big) \big]  + n\Big(\int_{0}^{\nicefrac{u_2}{\sqrt{n}}}G'(es_{\alpha}+s)s\,\dif s\Big) + o_{\mathcal{P}}(1)\notag\\[5pt]
		&= u_2 \,G'(es_{\alpha}) \frac{1}{\sqrt{n}} \sum_{i=1}^{n} \big(es_{\alpha}-q_{\alpha}+\alpha^{-1}1_{X_i\leq q_{\alpha}}(q_{\alpha}-X_i)\big)+ \int_{0}^{u_2 }G'(es_{\alpha}+t/\sqrt{n})t\,\dif t + o_{\mathcal{P}}(1).\label{eq:expansionprocess2}
		\end{align}
		The first term converges weakly by the central limit theorem to $G'(es_{\alpha})\, u_2\, W_2$ for the stated $W_2$, as $ \big(es_{\alpha}-q_{\alpha}+\alpha^{-1}1_{X_i\leq q_{\alpha}}(q_{\alpha}-X_i) \big)_{i\in\N}$ is a sequence of i.i.d. random variables. The second term converges to $G'(es_{\alpha})\, u_2^2/2$, thus the limit process $U$ of $U_n$ has the asserted form. 
		To conclude the proof of (\ref{eq:processconvlem1}), it remains to show (\ref{eq:l_cont_of_rest}). 
		Using (\ref{eq:diff_s_1}) we compute
		\begin{align*}
		&\big|S(q_{\alpha},x_2';y)-S(q_{\alpha},x_2;y)\big| = \big|S(q_{\alpha},x_2'+(x_2-x_2');y)-S(q_{\alpha},x_2';y)\big|\\
		&=\Big|\big(G(x_2)-G(x_2')\big)\big(x_2'-q_{\alpha}+\alpha^{-1}1_{y\leq q_{\alpha}}(q_{\alpha}-y)\big) + \int_{0}^{x_2-x_2'}G'(x_2'+s)s\,\dif s\Big|\\
		&\leq \big|G(x_2)-G(x_2')\big|\big(c_0 + |q_{\alpha}| +\alpha^{-1}1_{y\leq q_{\alpha}}(q_{\alpha}-y)\big) + \Big|\int_{0}^{x_2-x_2'}G'(x_2'+s)s\,\dif s\Big|.
		\end{align*}
		It follows from the mean value theorem, that we can find a $\xi\in K_2$ for which  $\big|G(x_2)-G(x_2')\big|=\big|G'(\xi)(x_2-x_2')\big|$. The right hand side of this is smaller than $C |x_2-x_2'|$, as $G'$ is continuous and hence bounded on $K_2$ ($C=\sup_{x_2\in K_2} G'(x_2)$). This ends the discussion of the first addend above as we therefore obtain \begin{equation*}
		\big|G(x_2)-G(x_2')\big|\big(c_0 + |q_{\alpha}| +\alpha^{-1}1_{y\leq q_{\alpha}}(q_{\alpha}-y)\big) \leq C \big(c_0 + |q_{\alpha}| +\alpha^{-1}1_{y\leq q_{\alpha}}(q_{\alpha}-y)\big) |x_2-x_2'|.
		\end{equation*} 
		For the other addend we utilize the (second) mean value theorem to get \begin{align*}
		\Big|\int_{0}^{x_2-x_2'} G'(x_2'+s)s\,\dif s\Big| &= \big|G'(x_2'+\xi)\xi (x_2-x_2')\big| \qquad \text{for some }\xi\in [-2c_2,2c_2]\\
		&\leq C' c_2 |x_2-x_2'|,
		\end{align*}
		where the inequality holds because $x_2\mapsto G'(x_2)x_2$ is continuous and thus bounded on $K_2$. All in all we end up with \begin{equation*}
		\big|S(q_{\alpha},x_2';y)-S(q_{\alpha},x_2;y)\big| \leq\big(C \big(c_0 + |q_{\alpha}| +\alpha^{-1}1_{y\leq q_{\alpha}}(q_{\alpha}-y)\big) +  C' c_2 \big) |x_2-x_2'| =: m(y)|x_2-x_2'|.
		\end{equation*}
		Now under [B] it holds that $\mathbb{E}\big[1_{Y\leq 0}Y^2\big]<\infty$, so in this case $\mathbb{E}[m(Y)^2]<\infty$ is true. 

		Finally, if Assumptions [A] and [B] are both satisfied, then the expansions (\ref{eq:expansionprocess1}) and (\ref{eq:expansionprocess2}) hold true and the joint process convergence follows, where the covariance of $W_1$ and $W_2$ is easily computed. 
	\end{proof}

	\begin{proof}[{\sl Proof of Lemma \ref{lem:approxminimizer}}]
		The process $U$ is quadratic and has the unique minimizer $u_2^0 = - W_2$. Further, from the form (\ref{eq:provessun}) of $U_n$ and the argument leading to (\ref{eq:esscoringexpl}) it follows that the unique minimizer of $U_n$ is given by
		\[ u_{2,n} = \sqrt{n}\,\big(\alpha^{-1} \mathbb{E}_n \big[1_{Y \leq q_\alpha}\, \big(Y-q_\alpha\big)  \big] - es_\alpha + q_\alpha \big),\]  
		which, by the central limit theorem, converges in distribution to $- W_2$. 
		
		To show (\ref{eq:minim_op_one}) we apply Lemma \ref{lem:diff_between_minimizer} to the processes
		\begin{align*}
		M_n(u_2)=U_n(u_2)+V_n(a_n(\widehat{q}_{n,\alpha} - q_{\alpha}))\big(1+\alpha^{-1}G(es_{\alpha}+u_2/\sqrt{n})\big),
		\end{align*}
		which is minimzed by $\sqrt{n}(\widehat{es}_{n,\alpha}-es_{\alpha})$, see (\ref{eq:esscoringexpl}), and  
		\begin{align*}
		M_n'(u_2)=U_n(u_2)+V_n(a_n(\widehat{q}_{n,\alpha} - q_{\alpha}))\big(1+\alpha^{-1}G(es_{\alpha})\big),
		\end{align*}
		so that $U_n$ will play the role of $N_n$ in Lemma \ref{lem:diff_between_minimizer}, which converges on compact sets to $U$ by Lemma \ref{eq:processconvlem2}.  
		Now, in Lemma \ref{cor:prep_es_tight} we showed that $\sqrt{n}(\widehat{es}_{n,\alpha}-es_{\alpha})$ is a tight sequence, and in the beginning of this proof we already showed that for the minimizers, $u_{2,n} \Rightarrow u_2^0$.  

		It thus remains to show that (\ref{eq:argmaxunifapprox}) holds true for the above choices of $M_n(u_2)$ and $M_n'(u_2)$,
		\begin{align}\label{eq:approxprocessunif}
		&\sup_{u_2\in K} \Big| U_n(u_2)+V_n(a_n(\widehat{q}_{n,\alpha} - q_{\alpha}))\big(1+\alpha^{-1}G(es_{\alpha}+u_2/\sqrt{n})\big)\nonumber\\
		&\qquad\qquad -\Big( U_n(u_2)+V_n(a_n(\widehat{q}_{n,\alpha} - q_{\alpha}))\big(1+\alpha^{-1}G(es_{\alpha})\big) \Big)\Big|=o_{\mathcal{P}}(1).
		\end{align}	
		To this end, assume $K\subset [-c_0,c_0]$. Then 
		\begin{align*}
		\sup_{u_2\in K} &\Big| U_n(u_2)+V_n(a_n(\widehat{q}_{n,\alpha} - q_{\alpha}))\big(1+\alpha^{-1}G(es_{\alpha}+u_2/\sqrt{n})\big)\\
		&\qquad -\Big( U_n(u_2)+V_n(a_n(\widehat{q}_{n,\alpha} - q_{\alpha}))\big(1+\alpha^{-1}G(es_{\alpha})\big) \Big)\Big|\\
		&=\alpha^{-1}c_0\Big|\frac{1}{\sqrt{n}}V_n(a_n(\widehat{q}_{n,\alpha} - q_{\alpha}))\Big|\sup_{u_2\in K}\Big|\frac{G(es_{\alpha}+u_2/\sqrt{n})-G(es_{\alpha})}{c_0/\sqrt{n}}\Big|\\
		&\leq \alpha^{-1}c_0\big|\frac{1}{\sqrt{n}}V_n(a_n(\widehat{q}_{n,\alpha} - q_{\alpha}))\big|\Big|\frac{G(es_{\alpha}+c_0/\sqrt{n})-G(es_{\alpha}-c_0/\sqrt{n})}{c_0/\sqrt{n}}\Big|
		\end{align*}
		since $G$ is monotonically non-decreasing. The first two factors are constant, the last factor is $O(1)$ since the fraction converges to $2G'(es_{\alpha})$, and it remains to show that $V_n\big(a_n(\widehat{q}_{n,\alpha}-q_{\alpha})\big) = o_{\mathcal{P}}(\sqrt{n})$, which is implied by
		\begin{equation}\label{eq:hilfapproxtight}
		V_n\big(a_n(\widehat{q}_{n,\alpha}-q_{\alpha})\big) = O_{\mathcal{P}}\big(\sqrt{n}/a_n\big).
		\end{equation}
		To see this, we first remark that $a_n(\widehat{q}_{n,\alpha}-q_{\alpha})$ is a tight sequence. This follows from the results in \citet{Knight2002LimDistr}, but is directly implied by (\ref{eq:processconvlem1}), convexity of the $V_n$ and $V$, and uniqueness of the minimizer of $V(u_1)$ with the aid of Lemma 2.2 in \citet{davis1992mestimation}. 
		Thus given $\varepsilon>0$ there exists a compact set $K_1$ with $\mathcal{P}\big(a_n(\widehat{q}_{n,\alpha}-q_{\alpha}) \in K_1\big) \geq 1-\varepsilon$. 
		Since for fixed compact $K_1\subset \R$ the map $f\mapsto \inf_{K_1} f$ is continuous for $f\in\ell^{\infty}(K_1)$ (w.r.t. the sup-norm), (\ref{eq:processconvlem1}) implies that $\inf_{K_1} \big(a_n/\sqrt{n}\big) \,V_n \Rightarrow \inf_{K_1} V$, in particular $\inf_{K_1} \big(a_n/\sqrt{n}\big) \,V_n$ is a tight sequence. To conclude note that \begin{align*}
		\mathcal{P}\Big(\frac{a_n}{\sqrt{n}} \,V_n(a_n(\widehat{q}_{n,\alpha}-q_{\alpha})) \geq C\Big) \leq \mathcal{P}\Big(\inf_{K_1} \frac{a_n}{\sqrt{n}}\,V_n \geq C\Big) + \mathcal{P}\big(a_n(\widehat{q}_{n,\alpha}-q_{\alpha}) \notin K_1\big),
		\end{align*}
		which implies (\ref{eq:hilfapproxtight}), and finishes the proof of the lemma. 
		
	\end{proof}
}



\bibliographystyle{chicago}
\nocite{*}
\bibliography{ExpectedShortfall}

%% file: supplement.tex
{\title{Asymptotics for the expected shortfall: Technical Supplement}}
\author{}
\maketitle

\begin{abstract}
We provide additional details to the paper \citet{zwingHolz2016}.
\end{abstract}

\section{Notation and results from the main paper}
For i.i.d.~observations $X_1, \ldots, X_n$ distributed according to the distribution function $F$, we use the notation \begin{equation*}
 \mathbb{E}_n [f(Y)]  =\frac{1}{n}\sum_{i=1}^{n} f(X_i) \qquad \text{ and } \quad \mathbb{G}_n [f(Y)] = \sqrt{n}(\mathbb{E}_n-\mathbb{E})[f(Y)]
\end{equation*}
where $E |f(X_1)| < \infty$. Note that $\mathbb{E}_n$ is the empirical expectation w.r.t. $X_1, \ldots , X_n$. 
Further, we let $X_{1:n} \leq \ldots \leq X_{n:n}$ denote the order statistics of a sample $X_1, \ldots, X_n$. 
We denote by $\rightarrow$ convergence in distribution. 

Suppose that the random variable $Y$ has distribution function $F$ and satisfies $\mathbb{E}[ |Y|] < \infty$. Given $ \alpha \in (0,1)$ the lower tail expected shortfall of $Y$ at level $\alpha$ is defined by
\[ es_\alpha = \frac{1}{\alpha}\, \int_0^\alpha F^{-1}(u)\, du.\]
For the specific value of $\alpha$ under consideration, we shall always impose the following. 
\medskip

{\bf Assumption.}\quad For the given $\alpha \in (0,1)$, the distribution function $F$ is continuous and strictly increasing at its $\alpha$-quantile $q_\alpha$.
Under this assumption, $F$ has a unique $\alpha$-quantile. 

Further, for the expected shortfall we have that

\begin{equation}\label{releq:lowertailes}
 es_\alpha = \frac{1}{\alpha}\, \int_{- \infty}^{q_\alpha}\, y\, dF(y).
\end{equation}
Consider the class of strictly consistent scoring functions for the bivariate parameter $(q_\alpha, es_\alpha)$ as introduced by \citet{fisszieg2015elicitability},
\begin{align}\label{releq:s_rep_1}
\begin{split}
	S(x_1, x_2; y) & = \big(1_{ y \leq x_1} - \alpha \big) (x_1-y) + G(x_2) \, \big(x_2 + \alpha^{-1}\big(1_{ y \leq x_1}-\alpha\big) (x_1 - y)\,   \big)\\
	& \qquad - \mathcal{G}(x_2) - G(x_2)y,
\end{split}
\end{align}
where $\mathcal{G}$ is a three-times continuously differentiable function, $\mathcal{G}' = G$ and it is required that $G' >0$. Further, from the proof of Corollary 5.5 in \citet{fisszieg2015elicitability} it follows that one may choose $\mathcal{G}$ so that $\lim_{x\to -\infty}G(x)=0$. We may write
\begin{align}\label{releq:s_rep_2}
\begin{split}
	S(x_1, x_2; y) &=\big(1+\alpha^{-1}G(x_2)\big)\, \big(1_{ y \leq x_1} - \alpha \big) (x_1-y) + G(x_2) \, \big(x_2 -y\big) - \mathcal{G}(x_2)
\end{split}
\end{align}

Denote the asymptotic contrast function	by
\[	S(x_1, x_2; F) = \mathbb{E} [S(x_1, x_2; Y)].\]

	Let $X_1, \ldots , X_n$ be i.i.d., distributed according to $F$ with $ \mathbb{E} |X_1| < \infty$. Consider the minimum contrast estimator for the bivariate parameter $(q_\alpha, es_\alpha)$ defined by
	\begin{align*}
		(\widehat{q}_{n,\alpha}, \widehat{es}_{n,\alpha}) &\in \argmin_{(x_1,x_2)\in\R^2} \sum_{i=1}^{n} S(x_1,x_2;X_i).
\end{align*}

%

\textbf{Assumption [A]}: There exists a function $\psi_{\alpha}:\R\to\overline{\R}$ with 
\begin{align*}
\lim_{t\to\infty} \psi_{\alpha}(t)=\infty, \qquad 
\lim_{t\to -\infty}\psi_{\alpha}(t)=-\infty
\end{align*}
such that for some deterministic, positive sequence $(a_n)_n$ with $a_n \to \infty$ it holds that \begin{equation*}
\lim_{n\to\infty} \sqrt{n}\big[F(q_{\alpha}+t/a_n)-F(q_{\alpha}) \big] = \psi_{\alpha}(t).
\end{equation*}

\vspace{12pt}
\textbf{Assumption [B]}: It holds that $\mathbb{E} \big[1_{Y \leq 0}\, Y^2\big]<\infty$.
\vspace{12pt}

\setcounter{theorem}{2}
\begin{theorem}\label{relthm:main_thm_es}
	
	Under Assumptions [A] and [B], we have that 
\[ \big(a_n(\widehat{q}_{n,\alpha}-q_{\alpha}), \sqrt{n}(\widehat{es}_{n,\alpha}-es_{\alpha})\big) \Rightarrow \big(\psi^{\leftrightarrow}_{\alpha}( W_1), W_2 \big),\]
	where  
	\begin{equation*}
	 (W_1,W_2)\sim \mathcal{N}(0,\Sigma), \qquad 	\Sigma=\begin{pmatrix}
	 	\alpha(1-\alpha) && (1-\alpha)(q_{\alpha}-es_{\alpha})\\
	 	(1-\alpha)(q_{\alpha}-es_{\alpha}) && \Var\big(1_{Y\leq q_{\alpha}}(q_{\alpha}-Y)/\alpha\big)
	 	\end{pmatrix}
	 \end{equation*}
and
 \begin{equation}\label{releq:formulainvfunct}
	\psi_{\alpha}^{\leftrightarrow}(x)=\begin{cases}
	\inf\{t\leq 0 \,|\, \psi_{\alpha}(t)\geq x\} &\mbox{if } x<0\\
	0 &\mbox{if } x=0\\
	\sup\{t\geq 0 \,|\, \psi_{\alpha}(t)\leq x\} &\mbox{if } x>0.
	\end{cases}
	\end{equation}
	
\end{theorem}

\textbf{Assumption [A$^k$]:} For each $m\in\{1, \ldots, k\}$ and corresponding $\alpha_m$ and $q_{\alpha_m}$, Assumption [A] is satisfied with associated sequence $(a_{m,n})_n$ and function $\psi_{\alpha_m}(t)$.

\begin{theorem}\label{relcor:conv_estimator_mult}
Let [A$^k$] and [B] hold. 	
	Then 
\begin{align*}
	\Big(a_{1,n} \big( \widehat{q}_{n,\alpha_1} - q_{\alpha_1}\big), \sqrt{n}\, \big(\widehat{es}_{n,\alpha_1} - es_{\alpha_1}\big),&\ldots , a_{k,n} \big( \widehat{q}_{n,\alpha_k} - q_{\alpha_k}\big), \sqrt{n}\, \big(\widehat{es}_{n,\alpha_k} - es_{\alpha_k}\big)\Big) \\[5pt]
	&\Rightarrow (z_{1,1}, z_{1,2}, \ldots , z_{k,1}, z_{k,2}),
\end{align*}
where for $m=1, \ldots, k$, 
$		z_{m,1} = \psi^{\leftrightarrow}_{\alpha_m}( W{m,1}) $ and $ 
		z_{m,2} = W_{m,2}$,
with $\psi^{\leftrightarrow}_{\alpha_m}$ as in (\ref{releq:formulainvfunct})	and the vector $\big(W_{1,1},W_{1,2}, \ldots , W_{k,1},W_{k,2}\big)$ distributed according to $\mathcal{N}\big(0,\Sigma\big)$ with $\Sigma$ determined by \begin{align*}
		\Cov \big(W_1^s,W_1^t\big) &= \alpha_s\wedge \alpha_t - \alpha_s \alpha_t,\\[5pt]
		\Cov \big(W_2^s, W_2^t) &= \frac{\alpha_s\wedge\alpha_t}{\alpha_s\alpha_t}\big(q_{\alpha_s}q_{\alpha_t} - (q_{\alpha_s}+q_{\alpha_t})es_{\alpha_s\wedge\alpha_t} \big) + \frac{1}{\alpha_s\alpha_t} \mathbb{E}\big[1_{Y\leq q_{\alpha_s\wedge \alpha_t}}Y^2\big],\\
		& \quad + \big(es_{\alpha_s} - q_{\alpha_s} \big) \big(es_{\alpha_t} - q_{\alpha_t} \big)\\[5pt]
		\Cov(W_1^s,W_2^t) &= \frac{\alpha_s\wedge\alpha_t}{\alpha_t}\big(q_{\alpha_t}-es_{\alpha_s\wedge \alpha_t}\big) - \alpha_s\big(q_{\alpha_t}-es_{\alpha_t}\big)
	\end{align*}
	for $s,t\in\{1,\ldots , k\}$. 	
\end{theorem}


For a probability measure $\mu$ on $[0,1]$, called the \emph{spectral measure}, 
\begin{equation*}
		\nu_{\mu}(F) = \int_{[0,1]} es_{\alpha} \,\dif \mu(\alpha)
	\end{equation*}
	is called the \emph{spectral risk measure} associated to $\mu$. Here, the boundary cases are given by $es_1 = E [Y]$ and $es_0 =$ essinf$\, Y$. 
If $\mu$ is finitely supported in $(0,1)$, $\nu_{\mu}(F)$ is a finite convex combination of expected shortfalls for different levels,
\begin{equation*}
	\nu_{\mu} = \sum_{m=1}^{k}p_m \, es_{\alpha_m} \quad \text{if} \quad \mu = \sum_{m=1}^{k} p_m \delta_{\alpha_m}.
\end{equation*}
\citet{fisszieg2015elicitability} show that strictly consistent scoring functions for $\nu_{\mu}$ in this case are given by
\begin{align*}
	S_{sp}(x_1, \ldots , x_k, x_{k+1};z) &= \sum_{m=1}^{k} \Bigg(\Big(1+\frac{p_m}{\alpha_m}G(x_{k+1})\Big)\big(1_{z\leq x}-\alpha\big)\,(x-z) \\
	&\qquad\qquad + p_m\big(G(x_{k+1})(x_{k+1}-z) - \mathcal{G}(x_{k+1}\big)\big)\Bigg). 
\end{align*}
where the functions $G$ and $\mathcal{G}$ are as above. 
If we define the corresponding M estimator
\[
\big(\widehat{q}_{n,\alpha_1}, \ldots , \widehat{q}_{n,\alpha_k}, \widehat{\nu}_{\mu,n}\big) \in \argmin_{x_1, \ldots , x_{k+1}}\, S_{sp}(x_1, \ldots , x_{k+1}; Y_i),
\]
then we have the following result. 
\begin{theorem}\label{relth:estspectralrisk}
We have that 
\begin{equation}\label{releq:specrisk}
\widehat{\nu}_{\mu,n} = \sum_{m=1}^{k}p_m\, \widehat{es}_{n,\alpha_m}.
\end{equation}
Consequently, under Assumptions [A$^k$] and [B] we have that
\[ \sqrt{n}\, \big(\widehat{\nu}_{\mu,n} - \widehat{\nu}_{\mu} \big) \Rightarrow \sum_{m=1}^k p_m \, W_{m,2},\] 
where the $W_{m,2}$ are as in Theorem \ref{relcor:conv_estimator_mult}. 
\end{theorem}
%


\section{Missing details in the proof of Theorem \ref{relthm:main_thm_es}}
%


\setcounter{theorem}{8}
\begin{lemma}\label{rellem:dif_s_reformulated}
a.		We have that
\begin{align}
		&S(x_1,x_2+y_2; z) - S(x_1,x_2;z) \notag\\
		&=   \big(G(x_2+y_2)- G(x_2) \big)\big( x_2 - x_1 +\alpha^{-1} 1_{ z \leq x_1}\big(x_1-z \big)\big) + \int_{0}^{y_2} G'(x_2+s)s \,\dif s\label{releq:diff_s_1}\\
		&=\big(G(x_2+y_2)- G(x_2) \big) \big( x_2 - x_1 +\alpha^{-1} 1_{ z \leq x_1}\big(x_1-z \big)\big) \notag\\
		&\quad +  \frac{1}{2}G'(x_2+y_2) y_2^2 - \frac{1}{2} \int_{0}^{y_2} G''(x_2+s)s^2 \,\dif s\label{releq:diff_s_2}.
	\end{align}
b. Setting $\rho_{\alpha}(x_1;z) = \big(1_{z\leq x_1}-\alpha\big)(x_1-z)$ we have that 
\begin{equation}\label{releq:incrementquant}
		\rho_{\alpha}(x_1+y_1; z) - \rho_{\alpha}(x_1;z) = y_1\big(1_{z \leq x_1}-\alpha\big) + \int_0^{y_1} \big(1_{z\leq x_1+s}-1_{z\leq x_1}\big) \, \dif s.
	\end{equation}
		
c. Generally, we have that
\begin{align}
	S&(x_1+y_1, x_2+y_2;z) - S(x_1,x_2;z)\notag\\
		&= \big(1+\alpha^{-1}G(x_2+y_2)\big) \Bigg(y_1\big(1_{z\leq x_1}-\alpha\big) + \int_0^{y_1}1_{z\leq x_1+s}-1_{z\leq x_1} \, \dif s\Bigg)\notag\\[5pt]
		&\quad + \big(G(x_2+y_2)-G(x_2)\big)\big( x_2 - x_1 +\alpha^{-1} 1_{ y \leq x_1}\big(x_1-y \big)\big) \notag\\[5pt]
		&\quad +\frac{1}{2}G'(x_2+y_2)y_2^2 - \frac{1}{2}\int_0^{y_2}G''(x_2+s)s^2\,\dif s\label{releq:difference_s}.
\end{align}
		
\end{lemma}

\begin{proof}[{\sl Proof of Lemma \ref{rellem:dif_s_reformulated}}]
From (\ref{releq:s_rep_2}), we have that
\begin{align}\label{relhilfslemma11}
	&S(x_1+y_1, x_2+y_2;z) - S(x_1,x_2;z) \nonumber \\
	&= \big(1+\alpha^{-1}G(x_2+y_2)\big)\,\big(1_{z\leq x_1+y_1}-\alpha\big)(x_1+y_1-z) + G(x_2+y_2)\big(x_2+y_2-z\big) \nonumber\\
	&\quad - \mathcal{G}(x_2+y_2) -\big(1+\alpha^{-1}G(x_2)\big)\, \big(1_{ z\leq x_1} - \alpha \big) (x_1-z) - G(x_2) \, \big(x_2 -z\big) + \mathcal{G}(x_2) \nonumber\\
	&= \big(1+\alpha^{-1}G(x_2+y_2)\big)\,\big(1_{z\leq x_1+y_1}-\alpha\big)(x_1+y_1-z) - \big(1_{z\leq x_1} - \alpha\big)(x_1-z) \nonumber\\
	&\quad -\alpha^{-1}G(x_2)\, \big(1_{ z \leq x_1} - \alpha \big) (x_1-z)  \nonumber\\
	&\quad + G(x_2+y_2)\big(x_2+y_2-z\big) - \mathcal{G}(x_2+y_2) - G(x_2) \, \big(x_2 -z\big) + \mathcal{G}(x_2) \nonumber\\
 &  = I) + II),
\end{align}
where 
\begin{align*}
	I ) &= \big(1+\alpha^{-1}G(x_2+y_2)\big)\,\big(\rho_{\alpha}(x_1+y_1;z) - \rho_{\alpha}(x_1;z)\big), \\[5pt]
	II ) & =  \alpha^{-1}G(x_2+y_2)\, \big(1_{ z \leq x_1} - \alpha \big) (x_1-z) -\alpha^{-1}G(x_2)\, \big(1_{ z \leq x_1} - \alpha \big) (x_1-z)  \\[5pt]
	&\quad + G(x_2+y_2)\big(x_2+y_2-z\big) - \mathcal{G}(x_2+y_2) - G(x_2) \, \big(x_2 -z\big) - \mathcal{G}(x_2),
\end{align*}
and where we subtracted and added the term $\big(1_{z\leq x_1}-\alpha\big)(x_1-z)\alpha^{-1}G(x_2+y_2)$ in (\ref{relhilfslemma11}). 

 Observe that $I) = 0$ for $y_1=0$ and hence
\begin{equation}\label{releq:prep_un_diff_S}
	II) = S(x_1,x_2+y_2; z) - S(x_1,x_2;z).
\end{equation}
In $II)$ the terms $G(x_2+y_2)z$ and $G(x_2)z$ cancel out. Rearranging gives 
\begin{align*}
	II) &=  G(x_2+y_2)\, \big(\alpha^{-1} 1_{ z \leq x_1} - 1 \big) x_1 - G(x_2)\, \big(\alpha^{-1} 1_{ z \leq x_1}- 1 \big) x_1\\[5pt]
	&\quad - G(x_2+y_2)\alpha^{-1}1_{z\leq x_1} z + G(x_2)\alpha^{-1}1_{z\leq x_1}z \\[5pt]
	&\quad + G(x_2+y_2)x_2- G(x_2)x_2 +G(x_2+y_2)y_2-\mathcal{G}(x_2+y_2) + \mathcal{G}(x_2)\\
&= \big( \big(\alpha^{-1} 1_{ z \leq x_1} - 1 \big) x_1 - \alpha^{-1}1_{z\leq x_1} z + x_2\big) \big(G(x_2+y_2)- G(x_2) \big)\\[5pt]
		&\quad +G(x_2+y_2)y_2-\mathcal{G}(x_2+y_2) + \mathcal{G}(x_2)\\[5pt]
		&= \big( x_2 - x_1 +\alpha^{-1} 1_{ z \leq x_1}\big(x_1-z \big)\big) \big(G(x_2+y_2)- G(x_2) \big)\\[5pt]
		&\quad +G(x_2+y_2)y_2-\mathcal{G}(x_2+y_2) + \mathcal{G}(x_2).
\end{align*}
By a partial integration \begin{equation*}
	G(x_2+y_2)y_2-\mathcal{G}(x_2+y_2) + \mathcal{G}(x_2) = \int_{0}^{y_2} G'(x_2+s)s \,\dif s,
\end{equation*}
which together with (\ref{releq:prep_un_diff_S}) implies (\ref{releq:diff_s_1}). A further partial integration gives (\ref{releq:diff_s_2}). 

To prove (\ref{releq:incrementquant}), note that by a partial integration, 
\begin{align*}
	 \rho_{\alpha}(x_1+y_1;z) - \rho_{\alpha}(x_1;z) & = \int_{x_1}^{x_1 + y_1} \big(1_{z\leq s}-\alpha \big)\,\dif s = \int_{0}^{y_1} \big(1_{z\leq x_1+s}-\alpha\big) \,\dif s\\
		& = -\alpha\,y_1 + \int_{0}^{y_1} 1_{z\leq x_1+s}\,\dif s \\
& = y_1\big(1_{z\leq x_1}-\alpha\big) + \int_0^{y_1} \big(1_{z\leq x_1+s}-1_{z\leq x_1}\big) \, \dif s
\end{align*}
where in the last equality we added and subtracted the term $y_1 1_{z\leq x_1}$.  

Finally, combining (\ref{relhilfslemma11}), (\ref{releq:diff_s_2}), (\ref{releq:prep_un_diff_S}) and (\ref{releq:incrementquant}) gives (\ref{releq:difference_s}).

\end{proof}

\begin{lemma}\label{relcor:prep_es_tight}
	Assume $q_{n}$ to be a consistent estimator of $q_{\alpha}$ and [B] to hold, then the sequence $\sqrt{n}(\widetilde{es}_{n, \alpha}-es_{\alpha})$ is tight, where $\widetilde{es}_{n, \alpha}$ is the minimizer of the function \begin{equation*}
	x_2 \mapsto\sum_{i=1}^{n} S(q_{n}, x_2; X_i) = n\mathbb{E}_n[S(q_n,x_2;Y)].
	\end{equation*}
In particular, if [A] and [B] hold, then $\sqrt{n}\big(\widehat{es}_{n,\alpha}-es_{\alpha}\big)$ is a tight sequence. 
\end{lemma}

\begin{proof}[{\sl Proof of Lemma \ref{relcor:prep_es_tight}}]

We shall check the assumptions of Theorem~\ref{relthm:nuisance_rate_of_conv} with $\alpha=2$, $\beta=1$, $d_0, d_1$ the Euclidean distance in $\R$ and the criterion function $m_{\eta, \vartheta}(y) = S(\eta, \vartheta; y)$. Consistency has been taken care of in Theorem~2 in the main paper.
	
Concerning (\ref{releq:convrategen1}) we need to prove that 
\begin{equation}\label{releq:helprate1tech2}
	\inf_{|\eta-q_{\alpha}|\leq \varepsilon} \inf_{\delta_0\geq |\vartheta-es_{\alpha}|\geq \delta} \mathbb{E}\left[S(\eta, \vartheta; Y)-S(\eta, es_{\alpha};Y)\right] \geq C\,\delta^2,
	\end{equation}
	$0 < \varepsilon \leq c_0 $, $0<\delta \leq \delta_0$ for some $C, \delta_0, c_0>0$.

	To this end, using (\ref{releq:diff_s_1}) in Lemma~\ref{rellem:dif_s_reformulated}, we get that \begin{align*}
	& \inf_{|\eta-q_{\alpha}|\leq \varepsilon} \inf_{\delta_0\geq |\varepsilon_2|\geq \delta} \mathbb{E}\left[S(\eta, es_{\alpha}+\varepsilon_2;Y)-S(\eta, es_{\alpha};Y)\right]\\
	=&\inf_{|\eta-q_{\alpha}|\leq \varepsilon} \inf_{\delta_0\geq |\varepsilon_2|\geq \delta} \Bigg[\big(G(es_{\alpha} + \varepsilon_2) - G(es_{\alpha})\big)\mathbb{E}\big(es_{\alpha} -\eta + \alpha^{-1}1_{Y\leq \eta}\big(\eta-Y\big)\big)\\
	&\qquad\qquad\qquad\qquad+\int_0^{\varepsilon_2}G'(es_{\alpha}+s)s\,\dif s \Bigg]\\
	=&\inf_{|\eta-q_{\alpha}|\leq \varepsilon} \inf_{\delta_0\geq |\varepsilon_2|\geq \delta} \Bigg[\big(G(es_{\alpha} + \varepsilon_2) - G(es_{\alpha})\big)\mathbb{E}\Big(es_{\alpha}-Y +\alpha^{-1} \Big(\big(1_{Y\leq \eta}-\alpha\big)\big(\eta-Y\big)\Big)\Big)\\
	&\qquad\qquad\qquad\qquad+\int_0^{\varepsilon_2}G'(es_{\alpha}+s)s\,\dif s \Bigg],\\
	&\geq \inf_{\delta_0\geq |\varepsilon_2|\geq \delta} \big(G(es_{\alpha} + \varepsilon_2) - G(es_{\alpha})\big) \inf_{|\eta-q_{\alpha}|\leq \varepsilon}\Big(es_{\alpha}-\mathbb{E}[Y] +\alpha^{-1} \mathbb{E}\Big[\big(1_{y\leq \eta}-\alpha\big)\big(\eta-y\big)\Big]\Big)\\
	&\qquad+\inf_{\delta_0\geq |\varepsilon_2|\geq \delta}\int_0^{\varepsilon_2}G'(es_{\alpha}+s)s\,\dif s.
	\end{align*} 
The function $\eta\mapsto \big(1_{y\leq \eta}-\alpha\big)\big(\eta-y\big)$ is convex, so that $\eta\mapsto \mathbb{E}\big(\big(1_{Y\leq \eta}-\alpha\big)\big(\eta-Y\big)\big)$ is convex as well, where the (unique) minimum is attained in $q_{\alpha}$ (score function of $q_{\alpha}$). But $es_{\alpha} -\mathbb{E}[Y] + \alpha^{-1}\mathbb{E}\big(\big(1_{Y\leq q_{\alpha}}-\alpha\big)\big(q_{\alpha}-Y\big)\big) \big) =0$ and thus the expression above is greater than (or equal to) 
	\begin{align*}
	  \inf_{\delta_0\geq |\varepsilon_2|\geq \delta} \int_0^{\varepsilon_2}G'(es_{\alpha}+s)s\,\dif s.
	\end{align*}
	The remaining integral is monotonically increasing (decreasing) for $\varepsilon_2>0$ ($<0$), whence the infimum is attained in $\pm\delta$. A partial integration then gives \begin{align*}
	\inf_{\delta_0\geq |\varepsilon_2|\geq \delta}  \int_0^{\varepsilon_2}G'(es_{\alpha}+s)s\,\dif s &=  \delta^2\Bigg(\frac{1}{2} G'(es_{\alpha}\pm\delta)- \delta^{-2}\frac{1}{2}\int_0^{\pm\delta}G''(es_{\alpha}+s)s^2 \,\dif s \Bigg)\\
	&\geq \frac{\delta^2}{2}\Bigg(G'(es_{\alpha} \pm \delta) - \delta^{-2}C\,\int_{0}^{\pm\delta}s^2 \,\dif s\Bigg)\\
	&\geq \frac{\delta^2}{2}\Big(G'(es_{\alpha} \pm \delta) - \delta \,C/3\Big)
	\end{align*}
	for $C=\sup_{x_2\in [es_{\alpha}-\delta_0,es_{\alpha}+\delta_0 ]} |G''(x_2)|<\infty$. Since $\lim_{\delta \to 0} G'(es_{\alpha}\pm \delta) - \delta \,C/3 = G'(es_{\alpha}) >0$ holds by assumption on $G'$ we can find $C'>0$ and $\delta_0>0$ with $G'(es_{\alpha}\pm \delta) - \delta \,C/3 \geq C'$ for every $\delta \leq \delta_0$. This proves (\ref{releq:helprate1tech2}). 
	
For (\ref{releq:convrategen2}) we require  
\begin{equation}\label{releq:helprate1tech}
	\mathbb{E}\Big[ \sup_{ \substack{\scriptscriptstyle |\eta-q_{\alpha}|\leq \varepsilon\\\scriptscriptstyle |\vartheta-es_{\alpha}|\leq \delta}} \Big|\sqrt{n}(\mathbb{E}_n-\mathbb{E})\big( S(\eta, \vartheta;Y)-S(\eta, es_{\alpha};Y)\big)\Big| \Big] \leq C \,\delta, \qquad 0 < \delta \leq \delta_0,
	\end{equation}
for some $\delta_0, C >0$. 	
	To see this inequality we use (\ref{releq:diff_s_1}) again and the fact that $G(es_{\alpha} + \varepsilon_2) - G(es_{\alpha})$ equals $\int_{0}^{\varepsilon_2}G'(es_{\alpha}+s)\,\dif s$ to obtain \begin{align*}
	&\mathbb{E}\left[\sup_{|\eta-q_{\alpha}|\leq \varepsilon} \sup_{|\varepsilon_2|\leq \delta} \big|\sqrt{n}(\mathbb{E}_n-\mathbb{E})\big(S(\eta, es_{\alpha}+\varepsilon_2;Y)-S(\eta,es_{\alpha};Y)\big)\big|\right]\\[5pt]
	&=\mathbb{E}\left[\sup_{|\eta-q_{\alpha}|\leq \varepsilon} \sup_{|\varepsilon_2|\leq \delta} \Big|\int_{0}^{\varepsilon_2}G'(es_{\alpha}+s)\,\dif s\,\sqrt{n}\big(\mathbb{E}_n-\mathbb{E}\big)\big(\alpha^{-1}1_{Y\leq \eta}(\eta-Y) \big) \Big|\right]
	\end{align*}
	(only the stochastic term remains). Since $G'>0$, the former expression does not exceed \begin{equation*}
	\int_{-\delta}^{\delta}G'(es_{\alpha}+s)\,\dif s\,\alpha^{-1}\mathbb{E}\left[\sup_{|\eta-q_{\alpha}|\leq \varepsilon}\big|\sqrt{n}(\mathbb{E}_n-\mathbb{E})(1_{Y\leq \eta}(\eta-Y))\big|\right].
	\end{equation*}
	Because the first integral fulfils \begin{equation*}
		\int_{-\delta}^{\delta}G'(es_{\alpha}+s)\,\dif s \leq 2\,\delta G'(\xi)
	\end{equation*}
	for some $\xi\in [es_{\alpha}-\delta_0, es_{\alpha} + \delta_0]$ by the mean value theorem, and $G'>0$, it is enough to show \begin{equation}\label{releq:expect_sup_bound}
		\mathbb{E}\left[\sup_{|\eta-q_{\alpha}|\leq \varepsilon}\big|\sqrt{n}(\mathbb{E}_n-\mathbb{E})(1_{Y\leq \eta}(\eta-Y))\big|\right] \leq K
	\end{equation}
	for some constant $K$ not depending on $\delta$. 
	
	To this end we will use a maximal inequality involving the bracketing integral (Definition in Chapter~19.2, \citet{vdv2000asympstat}). Observe for any $\eta\in [q_{\alpha}-\varepsilon, q_{\alpha}+\varepsilon]$ the inequality \begin{align*}
		\big|1_{x\leq \eta}(\eta-x)\big| &\leq 1_{x\leq q_{\alpha}+\varepsilon}(q_{\alpha}+\varepsilon-x).
	\end{align*}
Thus $1_{x\leq q_{\alpha}+\varepsilon}(q_{\alpha}+\varepsilon-x)=:H(x)$ is an envelope function for the (measurable) class of functions $\mathcal{H}:=\big\{x\mapsto 1_{x\leq \eta}(\eta-x) \,|\, \eta \in [q_{\alpha}-\varepsilon, q_{\alpha}+\varepsilon]\big\}$. Using Corollary~19.35, \citet{vdv2000asympstat}, we obtain \begin{equation*}
		\mathbb{E}\left[\sup_{|\eta-q_{\alpha}|\leq \varepsilon}\big|\sqrt{n}(\mathbb{E}_n-\mathbb{E})(1_{Y\leq \eta}(\eta-Y))\big|\right] \leq K_1\, \int_0^{C} \sqrt{\log N_{[\,]}(\delta,\mathcal{H},\Vert\cdot\Vert_{Y,2}) } \,\dif \delta
	\end{equation*}
	for some constant $K_1<\infty$ and $C=\Vert H\Vert_{Y,2}$, where $N_{[\,]}(\delta,\mathcal{H},\Vert\cdot\Vert_{Y,2})$ denotes the bracketing number with respect to the norm $\Vert f\Vert_{Y,2} = \big(\mathbb{E}[f(Y)^2]\big)^{\nicefrac{1}{2}}$ (see the beginning of Chapter~19.2, \citet{vdv2000asympstat}; note $C<\infty$ by [B]). Next observe that the class $\mathcal{H}$ fulfills a Lipschitz-condition, namely for any $\eta_1,\eta_2 \in [q_{\alpha}-\varepsilon, q_{\alpha}+\varepsilon]$ it holds that \begin{equation*}
		\big|1_{x\leq \eta_1}(\eta_1-x)-1_{x\leq \eta_2}(\eta_2-x)\big| \leq |\eta_1-\eta_2|.
	\end{equation*}
	As seen in Example~19.7, \citet{vdv2000asympstat}, there is a constant $K_2$ only depending on $\varepsilon$, such that the bracketing number satisfies \begin{equation*}
		N_{[\,]}(\delta,\mathcal{H},\Vert\cdot\Vert_{Y,2}) \leq \frac{2\,K_2\,\varepsilon}{\delta}
	\end{equation*}
	for any $0<\delta < 2\,\varepsilon$. Hence by partitioning the bracketing integral we are left with \begin{align*}
		\int_0^{C} \sqrt{\log N_{[\,]}(\delta,\mathcal{H},\Vert\cdot\Vert_{Y,2}) } \,\dif \delta &\leq \int_{0}^{2\,\varepsilon} \sqrt{\log \big(2\,K_2\,\varepsilon/\delta\big)} \,\dif \delta + K_3 \\
		&\leq \sqrt{2\,K_2\,\varepsilon} \,\int_{0}^{2\,\varepsilon} \delta^{-\nicefrac{1}{2}}\,\dif \delta + K_3\\
		&= 4\,\sqrt{K_2} \,\varepsilon + K_3. 
	\end{align*}
	Putting things together we have shown \begin{equation*}
			\mathbb{E}\left[\sup_{|\eta-q_{\alpha}|\leq \varepsilon}\big|\sqrt{n}(\mathbb{E}_n-\mathbb{E})(1_{Y\leq \eta}(\eta-Y))\big|\right] \leq K_1\big(4\,\sqrt{K_2} \,\varepsilon + K_3 \big),
	\end{equation*}
	what is (\ref{releq:expect_sup_bound}). 
\end{proof}


\subsection{Proofs of Theorems \ref{relcor:conv_estimator_mult} and \ref{relth:estspectralrisk}}

Before giving the proof of Theorem \ref{relcor:conv_estimator_mult}, we require some results from the proof of Theorem \ref{relthm:main_thm_es}. 

Write
\begin{align}\label{relprocessrepres}
\begin{split}
	&\sum_{i=1}^{n} S\big(q_{\alpha}+u_1/a_n, es_{\alpha}+u_2/\sqrt{n}; X_i)-S(q_{\alpha}, es_{\alpha}; X_i)\\
	=\, & \, \big(1+\alpha^{-1}G(es_{\alpha}+u_2/\sqrt{n})\big)V_n(u_1) + U_n(u_2),
\end{split}
\end{align}
where 
\begin{align}
	V_n(u_1) & = \frac{u_1}{a_n}\sum_{i=1}^{n}(1_{X_i\leq q_{\alpha}}-\alpha)+ \frac{1}{a_n}\int_0^{u_1}\Big(\sum_{i=1}^{n}1_{X_i\leq q_{\alpha}+t/a_n}-1_{X_i\leq q_{\alpha}}\Big) \, \dif t\notag \\
	 & = \sum_{i=1}^n \rho_{\alpha}\big(q_{\alpha}+u_1/a_n;X_i\big)-\rho_{\alpha}\big(q_{\alpha};X_i \big) \label{releq:provessvn}
\end{align}
and
\begin{align}
	U_n(u_2) & = \sqrt{n}\big(G(es_{\alpha}+u_2/\sqrt{n})-G(es_{\alpha})\big) \frac{1}{\sqrt{n}}\sum_{i=1}^{n}\big(es_{\alpha}-q_{\alpha}+\alpha^{-1}1_{X_i\leq q_{\alpha}}(q_{\alpha}-X_i)\big) \notag\\
	&\quad +\frac{u_2^2}{2}G'(es_{\alpha}+u_2/\sqrt{n})- \frac{1}{2\sqrt{n}}\int_0^{u_2} G''(es_{\alpha}+t/\sqrt{n})t^2\,\dif t\\
	& = \sum_{i=1}^{n} S(q_{\alpha},es_{\alpha}+\nicefrac{u_2}{\sqrt{n}};y) - S(q_{\alpha},es_{\alpha};y). \label{releq:provessun}
\end{align}

In Lemma 11 we showed that 

 \begin{align}
	\frac{a_n}{\sqrt{n}}V_n(u_1)&= \sqrt{n}\mathbb{E}_n\big[a_n\big(\rho_{\alpha}(q_{\alpha}+u_1/a_n;Y)-\rho_{\alpha}(q_{\alpha};Y)\big)\big]\notag\\
	&= \frac{u_1}{\sqrt{n}}\sum_{i=1}^{n}\big(1_{X_i\leq q_{\alpha}}-\alpha\big) + \int_0^{u_1}\sqrt{n}\big(F(q_{\alpha}+s/a_n)-F(q_{\alpha})\big)\,\dif s + o_{\mathcal{P}}(1).\label{releq:expansionprocess1}
	\end{align}

and that 

	\begin{align}
	&U_n(u_2)\notag\\
	&= u_2 \,G'(es_{\alpha}) \frac{1}{\sqrt{n}} \sum_{i=1}^{n} \big(es_{\alpha}-q_{\alpha}+\alpha^{-1}1_{X_i\leq q_{\alpha}}(q_{\alpha}-X_i)\big)+ \int_{0}^{u_2 }G'(es_{\alpha}+t/\sqrt{n})t\,\dif t + o_{\mathcal{P}}(1).\label{releq:expansionprocess2}
	\end{align}

\setcounter{theorem}{11}
\begin{lemma}\label{rellem:approxminimizer}
The processes $(U_n)$ and $U(u_2) = G'(es_{\alpha})\, \big(\, u_2\, W_2 + u_2^2\,  /2\big)$ have unique minimizers $({u}_{2,n})$ and $u_{2}^0$, and ${u}_{2,n} \Rightarrow u_{2}^0$. Moreover, we have that
	\begin{equation}\label{releq:minim_op_one}
	\sqrt{n}(\widehat{es}_{n,\alpha}-es_{\alpha})  = {u}_{2,n}+o_{\mathcal{P}}(1).
	\end{equation}
\end{lemma}

\begin{proof}[{\sl Proof of Theorem \ref{relcor:conv_estimator_mult}}]
We define the processes $V_n^m$ and $U_n^m$ as in (\ref{releq:provessvn}) and (\ref{releq:provessun}) for each $\alpha_m$, $m=1, \ldots, k$. Then the expansions (\ref{releq:expansionprocess1}) and (\ref{releq:expansionprocess2}) are valid for each $m$, and the covariance matrix in the joint normal distribution of the vector $\big(W_{1,1},W_{1,2}, \ldots , W_{k,1},W_{k,2}\big)$ in the limit processes $U^m$ and $V^m$, which are given by 
\[ 
V(u_1) =  u_1 W_1 + \int_0^{u_1} \psi_{\alpha}(t)\,\dif t,
\]
see Lemma 11 in the main paper,  is determined by 
\begin{align*}
	\mathbb{E} \big[\big(1_{Y\leq q_{\alpha_s}}-\alpha_s\big)\big(1_{Y\leq q_{\alpha_t}}-\alpha_t\big)\big]
=\alpha_s\wedge\alpha_t - \alpha_s\alpha_t,
	\end{align*}
	\begin{align*}
	&\mathbb{E}\Big[\big(es_{\alpha_s} - q_{\alpha_s}+\alpha_s^{-1} 1_{Y\leq q_{\alpha_s}}\big(q_{\alpha_s}-Y\big)\big)\big(es_{\alpha_t} - q_{\alpha_t}+\alpha_t^{-1} 1_{Y\leq q_{\alpha_t}}\big(q_{\alpha_t}-Y\big)\big) \Big]\\
		&=\alpha_t^{-1}\alpha_s^{-1} \mathbb{E}\Big[ 1_{Y\leq q_{\alpha_s}\wedge q_{\alpha_t}}\big(q_{\alpha_s}-Y\big) \big(q_{\alpha_t}-Y\big) \Big]
		 \ + \big(es_{\alpha_s} - q_{\alpha_s} \big) \big(es_{\alpha_t} - q_{\alpha_t} \big)\\
		&= \frac{\alpha_s\wedge\alpha_t}{\alpha_s\alpha_t}\big(q_{\alpha_s}q_{\alpha_t}-(q_{\alpha_s}+q_{\alpha_t})es_{\alpha_s\wedge\alpha_t}\big) + \alpha_t^{-1}\alpha_s^{-1} \mathbb{E}\big[1_{Y\leq q_{\alpha_s}\wedge q_{\alpha_t}}Y^2\big]\\
		& \quad + \big(es_{\alpha_s} - q_{\alpha_s} \big) \big(es_{\alpha_t} - q_{\alpha_t} \big)
	\end{align*}
	and
	\begin{align*}
& \mathbb{E}\Big[\big(1_{Y\leq q_{\alpha_s}}-\alpha_s\big)\big(es_{\alpha_t} - q_{\alpha_t}+\alpha_t^{-1} 1_{Y\leq q_{\alpha_t}}\big(q_{\alpha_t}-Y\big)\big)\Big]\\
		= &\, \alpha_t^{-1}\mathbb{E}\Big[\big(1_{Y\leq q_{\alpha_s}}-\alpha_s\big)\big(q_{\alpha_t}-Y\big)1_{Y\leq q_{\alpha_t}}\Big]\\
		= & \, \frac{\alpha_t\wedge\alpha_s}{\alpha_t}\big(q_{\alpha_t}- es_{\alpha_s\wedge \alpha_t}\big) - \alpha_s \big(q_{\alpha_t} -  es_{\alpha_t}\big)
	\end{align*}
where $s,t\in\{1,\ldots , k\}$. Further, Lemma \ref{rellem:approxminimizer} also holds true for each $m$. As in step 5 of the proof of Theorem \ref{relthm:main_thm_es} in the main paper, we may then consider the sequence of minimizers of the processes
\begin{equation*}
	Z_{n,mult} (v_1,u_1,\ldots , v_k,u_k) = \sum_{m=1}^{k}\big(\big(1+\alpha_m^{-1}G_i(es_{\alpha_m})\big) \frac{a_{m,n}}{\sqrt{n}} V_n^m(v_m) + U_n^m(u_m)\big),
\end{equation*}
which converge to the process
\begin{equation*}
	Z_{mult} (v_1,u_1,\ldots , v_k,u_k) = \sum_{m=1}^{k}\big(\big(1+\alpha_m^{-1}G_i(es_{\alpha_m})\big)  V^m(v_m) + U^m(u_m)\big),
\end{equation*}
and apply the argmax-continuity theorem to obtain the result. 
\end{proof}

\begin{proof}[{\sl Proof of Theorem \ref{relth:estspectralrisk}}]
The formula (\ref{releq:specrisk}) together with Theorem \ref{relcor:conv_estimator_mult} immediately imply the second statement of the theorem. 
Concerning (\ref{releq:specrisk}), setting

\begin{align*}
		g_m(x_{k+1}) = 1+\frac{p_m}{\alpha_m}G(x_{k+1}),  \quad 
		h(x_{k+1};z)=G(x_{k+1})(x_{k+1}-z)-\mathcal{G}(x_{k+1})
	\end{align*}
we have that
\begin{align*}
		\big(\widehat{q}_{n,\alpha_1}, \ldots , \widehat{q}_{n,\alpha_k}, \widehat{\nu}_{\mu,n}\big) &\in\argmin_{x_1, \ldots , x_{k+1} \in\R} \sum_{m=1}^{k}\sum_{i=1}^{n}\Big[g_m(x_{k+1})\rho_{\alpha_m}(x_m;Y_i) + p_m\,h(x_{k+1};Y_i)\Big].
	\end{align*}

	For the minimal value we have \begin{align*}
		&= \min_{x_{k+1}\in\R} \sum_{m=1}^{k}\Big[g_m(x_{k+1}) \Big(\min_{x_m\in\R} \sum_{i=1}^{n}\rho_{\alpha_m}(x_m;Y_i)\Big) + p_m\sum_{i=1}^{n}\,h(x_{k+1};Y_i)\Big],
	\end{align*}
	so the minimizer in $x_m$ does not depend on $x_l$, $l\ne m$ and is actually given by  
$		\widehat{q}_{n,\alpha_m}$. It remains to find the minimizer 
	of the function 
	\begin{equation*}
		x_{k+1}\mapsto \sum_{m=1}^{k}\Big[\sum_{i=1}^{n}g_m(x_{k+1})\rho_{\alpha_m}(\widehat{q}_{n,\alpha_m};Y_i) + p_m\,h(x_{k+1};Y_i)\Big].
	\end{equation*}
	Differentiation of the maps $g_m$ and $h$ gives \begin{align*}
		\partial_{x_{k+1}} g_m(x_{k+1}) = \frac{p_m}{\alpha_m}G'(x_{k+1}), \quad
		\partial_{x_{k+1}} h(x_{k+1};z) &= G'(x_{k+1})(x_{k+1}-z),
	\end{align*}
so that minimizing the above function is equivalent to solving \begin{align*}
		0
		&= G'(x_{k+1})\sum_{m=1}^{k} p_m\Big[ n\,x_{k+1}- n\,\widehat{q}_{n,\alpha_m} +\alpha_m^{-1}\sum_{i=1}^{n}1_{Y_i\leq\widehat{q}_{n,\alpha_m} }\big(\widehat{q}_{n,\alpha_m}-Y_i\big) \Big]
	\end{align*}
	for $x_{k+1}$, which results in 
	\begin{equation*}
		\widehat{\nu}_{\mu,n} = \sum_{m=1}^{k}p_m\Bigg[\widehat{q}_{n,\alpha_m} - \frac{1}{n\,\alpha_m}\sum_{i=1}^{n}1_{Y_i\leq \widehat{q}_{n,\alpha_m}}\big(\widehat{q}_{n,\alpha_m}-Y_i\big)\Bigg].
	\end{equation*}
The formula 
\begin{align}\label{releq:esscoringexpl}
\begin{split}		
		\widehat{es}_{n,\alpha} 
		& = \alpha^{-1}\mathbb{E}_n\big(Y \,1_{Y\leq \widehat{q}_{n,\alpha}}\big) + \widehat{q}_{n,\alpha} \, \big(1 - \frac{1}{\alpha\, n}\sum_{i=1}^{n}1_{X_i\leq \widehat{q}_{n,\alpha}} \big)
\end{split}
	\end{align}
for the expected shortfall then implies (\ref{releq:specrisk}). 
\end{proof}


\section{A general result on rates of convergence}

The rates of convergence will be proved using the next theorem, which is a generalization of Theorem~5.52, \citet{vdv2000asympstat}, and similar to his  Theorem 5.23. We will provide a proof for convenience. Assume that $(\Theta_0, d_0), (\Theta_1, d_1)$ are metric spaces and that for all $\eta\in \Theta_0$, $\vartheta\in\Theta_1$, the map $y\mapsto m_{\eta, \vartheta}(y)$ is measurable.  To unify notation, we will use $\mathbb{E}_n$ and $Y$ in the formulation of the theorem, but note that $Y$ here could also have a more general form (not needing a finite first moment or $Y$ to be real).

\begin{theorem}\label{relthm:nuisance_rate_of_conv}
	Assume that for fixed $C$ and $\alpha>\beta$, every $n\in\N$ and all sufficiently small $\varepsilon,\delta>0$ it holds that 
	\begin{align}\label{releq:convrategen1}
	\inf_{d_0(\eta, \eta_0)\leq \varepsilon} \inf_{d_1(\vartheta, \vartheta_0)\geq \delta} \mathbb{E}\big[m_{\eta,\vartheta}(Y) - m_{\eta, \vartheta_0}(Y)\big] \geq C\delta^{\alpha}
	\end{align}
	and 
	\begin{equation}\label{releq:convrategen2}
	\mathbb{E}^*\Big[\sup_{d_0(\eta,\eta_0)\leq \varepsilon}\sup_{d_1(\vartheta, \vartheta_0)\leq \delta} \big|\sqrt{n}(\mathbb{E}_n-\mathbb{E})\big(m_{\eta, \vartheta}(Y) - m_{\eta, \vartheta_0}(Y)\big)\big|\Big] \leq C\delta^{\beta}. 
	\end{equation}
	Additionally suppose that $\eta_n$ converges to $\eta_0$ in (outer) probability and $\widehat{\vartheta}_n$ converges to $\vartheta_0$ in (outer) probability and fulfils \begin{equation*}
	\mathbb{E}_n\big[m_{\eta_n,\widehat{\vartheta}_n}(Y)\big]\leq  \mathbb{E}_n\big[m_{\eta_n, \vartheta_0}(Y)\big] + O_{\mathcal{P}}(n^{\nicefrac{\alpha}{(2(\beta - \alpha))}}).
	\end{equation*}
	Then $n^{\nicefrac{1}{(2(\alpha-\beta))}}d_1(\widehat{\vartheta}_n,\vartheta_0) = O_{\mathcal{P}}^*(1)$. 
\end{theorem}

\begin{proof}[{\sl Proof of Theorem~\ref{relthm:nuisance_rate_of_conv}}]
	
	We set $r_n =n^{\nicefrac{1}{(2\alpha-2\beta)}}$ and suppose, that $\widehat{\vartheta}_n$ minimises the map $\vartheta\mapsto \mathbb{E}_n\big[m_{\eta_n,\vartheta}(Y)\big]$ up to a random variable $R_n = O_{\mathcal{P}}(r_n^{-\alpha})$. 
	
	For each $n$ the set $\Theta_1\setminus \{\vartheta_0\}$ can be partitioned into the sets \begin{equation*}
	S_{j,n}=\big\{ \vartheta \,|\, 2^{j-1}< r_n d_1(\vartheta, \vartheta_0) \leq 2^j \big\}, \quad j\in\Z. 
	\end{equation*} 
	If $r_n d_1(\widehat{\vartheta}_n, \vartheta_0) >2^M$ for some $M\in\Z$, then $\widehat{\vartheta}_{n}$ must be in one of the $S_{j,n}$ for $j\geq M$. Further, if $\rho >0$ and $d_1(\widehat{\vartheta}_n,\vartheta_0)\leq \rho/2$ then $\widehat{\vartheta}_n\in S_{j,n}$ for $2^j \leq \rho\, r_n$. This gives \begin{align*}
	\mathcal{P}^*\big(r_n d_1(\vartheta, \vartheta_0) >2^M\big) &\leq \mathcal{P}^*\Bigg(\bigcup_{\substack{\scriptscriptstyle j\geq M \\ \scriptscriptstyle 2^j\leq \rho\,r_n} } \{\widehat{\vartheta}_n\in S_{j,n} \}\cap \{d_1(\widehat{\vartheta}_n, \vartheta_0)\leq \rho/2\} \cap \{d_0(\eta_n,\eta_0)\leq \rho \}\Bigg)\\
	&\quad + \mathcal{P}^*\big(2d_1(\widehat{\vartheta}_n, \vartheta_0) >\rho\big) + \mathcal{P}^*\big(d_0(\eta_n,\eta_0)>\rho\big).
	\end{align*}
	
	Assume $\widehat{\vartheta}_n\in S_{j,n}$ for a $j$ involved in the above union. Then by assumption on $\widehat{\vartheta}_n$ the infimum of the map $\vartheta \mapsto \mathbb{E}_n\big[m_{\eta_n,\vartheta}(Y)-m_{\eta_n, \vartheta_0}(Y)\big]$ over $S_{j,n}$ is at most $R_n$. If we suppose in addition that $d_0(\eta_n,\eta_0)\leq \rho$ holds, then the infimum of $(\eta,\vartheta)\mapsto \mathbb{E}_n\big[m_{\eta,\vartheta}(Y)-m_{\eta, \vartheta_0}(Y) \big]$ over $B_{\rho}(\eta_0)\times S_{j,n}$ is smaller than $R_n$ as well. Hence if $r_n^{\alpha}R_n \leq C'$ for some $C'<\infty$, this infimum is smaller than $C'/r_n^{\alpha}$. Thus \begin{align}
	\mathcal{P}^*\big(r_n d_1(\vartheta, \vartheta_0) >2^M\big) &\leq \sum_{\substack{\scriptscriptstyle j\geq M \\ \scriptscriptstyle 2^j\leq \rho\,r_n} } \mathcal{P}^*\Big( \inf_{\eta \in B_{\rho}(\eta_0)} \inf_{\vartheta\in S_{j,n}} \mathbb{E}_n\big[m_{\eta,\vartheta}(Y)-m_{\eta, \vartheta_0}(Y) \big] \leq C'/r_n^{\alpha} \Big) \notag\\
	&\quad +  \mathcal{P}^*\big(2d_1(\widehat{\vartheta}_n, \vartheta_0) >\rho\big) + \mathcal{P}^*\big(d_0(\eta_n,\eta_0)>\rho\big) + \mathcal{P}^*\big( r_n^{\alpha}R_n > C'\big).\label{releq:every_is_small}
	\end{align}	
	Observe that the last three summands can be made small for any $\rho>0$ by choosing $n$ and $C'$ big enough ($\eta_n\to\eta$, $\widehat{\vartheta}_n\to\vartheta_0$ in (outer) probability, $R_n = O_{\mathcal{P}}(r_n^{-\alpha})$). 
	
	Now choose $\rho>0$ small enough to ensure that the conditions of the theorem hold for all $\delta, \varepsilon \leq \rho$. Every $j$ involved in the above sum does fulfils $2^{j} /r_n \leq \rho$, whence from the first assumption (\ref{releq:convrategen1})
	\begin{align}
	&\inf_{\eta \in B_{\rho}(\eta_0)} \inf_{\vartheta\in S_{j,n}} \mathbb{E}\big[m_{\eta,\vartheta}(Y)-m_{\eta, \vartheta_0}(Y)  \big]\notag\\
	&\geq \inf_{\eta \in B_{\rho}(\eta_0)} \inf_{\rho \geq d_1(\vartheta, \vartheta_0) \geq 2^{j-1}/r_n} \mathbb{E}\big[m_{\eta,\vartheta}(Y)-m_{\eta, \vartheta_0}(Y) \big]\notag\\
	&\geq C \Big(\frac{2^{j-1}}{r_n}\Big)^{\alpha}.\label{releq:add_at_least}
	\end{align}
	Hence 
	\begin{align*}
	&\sum_{\substack{\scriptscriptstyle j\geq M \\ \scriptscriptstyle 2^j\leq \rho\,r_n} } \mathcal{P}^*\Big( \inf_{\eta \in B_{\rho}(\eta_0)} \inf_{\vartheta\in S_{j,n}} \mathbb{E}_n\big[m_{\eta,\vartheta}(Y)-m_{\eta, \vartheta_0}(Y) \big] \leq C'/r_n^{\alpha} \Big)\\
	&= \sum_{\substack{\scriptscriptstyle j\geq M \\ \scriptscriptstyle 2^j\leq \rho\,r_n} } \mathcal{P}^*\Big( \inf_{\eta \in B_{\rho}(\eta_0)} \inf_{\vartheta\in S_{j,n}} \Big[(\mathbb{E}_n-\mathbb{E})\big[m_{\eta,\vartheta}(Y)-m_{\eta, \vartheta_0}(Y) \big] +\mathbb{E}\big[m_{\eta,\vartheta}(Y)-m_{\eta, \vartheta_0}(Y) \big]\Big]\leq C'/r_n^{\alpha} \Big)\\
	& \leq \sum_{\substack{\scriptscriptstyle j\geq M \\ \scriptscriptstyle 2^j\leq \rho\,r_n} } \mathcal{P}^*\Big( \inf_{\eta \in B_{\rho}(\eta_0)} \inf_{\vartheta\in S_{j,n}} \Big[(\mathbb{E}_n-\mathbb{E})\big[m_{\eta,\vartheta}(Y)-m_{\eta, \vartheta_0}(Y) \big]\Big] \\
	&\qquad\qquad\qquad\qquad\qquad\qquad+ \inf_{\eta \in B_{\rho}(\eta_0)} \inf_{\vartheta\in S_{j,n}} \mathbb{E}\big[m_{\eta,\vartheta}(Y)-m_{\eta, \vartheta_0}(Y) \big] \leq C'/r_n^{\alpha} \Big)\\[7pt]
	&\leq \sum_{\substack{\scriptscriptstyle j\geq M \\ \scriptscriptstyle 2^j\leq \rho\,r_n} } \mathcal{P}^*\Big( \inf_{\eta \in B_{\rho}(\eta_0)} \inf_{\vartheta\in S_{j,n}} (\mathbb{E}_n-\mathbb{E})\big[m_{\eta,\vartheta}(Y)-m_{\eta, \vartheta_0}(Y) \big) \leq \big(C'-C\,2^{(j-1)\alpha}\big)/r_n^{\alpha}\Big),
	\end{align*}
	where the last inequality uses (\ref{releq:add_at_least}). Choose $M$ large enough to guarantee $C'\leq C\,2^{(M-1)\alpha-1}$, so that \begin{equation*}
	C'- C\, 2^{(j-1)\alpha} \leq C\,2^{(j-1)\alpha-1}-C\,2^{(j-1)\alpha} = -C\,2^{(j-1)\alpha-1}
	\end{equation*} holds for $j\geq M$. This means that the former sum does not exceed \begin{equation*}
	\sum_{\substack{\scriptscriptstyle j\geq M \\ \scriptscriptstyle 2^j\leq \rho\,r_n} } \mathcal{P}^*\Big( \inf_{\eta \in B_{\rho}(\eta_0)} \inf_{\vartheta\in S_{j,n}} (\mathbb{E}_n-\mathbb{E})\big[m_{\eta,\vartheta}(Y)-m_{\eta, \vartheta_0}(Y) \big) \leq -C\,2^{(j-1)\alpha}/(2\,r_n^{\alpha})\Big).
	\end{equation*}
	By taking absolute values and multiplying with $\sqrt{n}$ this expression is smaller than \begin{equation*}
	\sum_{\substack{\scriptscriptstyle j\geq M \\ \scriptscriptstyle 2^j\leq \rho\,r_n} } \mathcal{P}^*\Big( \sup_{\eta \in B_{\rho}(\eta_0)} \sup_{\vartheta\in B_{2^j/r_n}(\vartheta_0)}\Big| \sqrt{n}(\mathbb{E}_n-\mathbb{E})\big[m_{\eta,\vartheta}(Y)-m_{\eta, \vartheta_0}(Y)\big]\Big| \geq C \sqrt{n} \frac{2^{(j-1)\alpha}}{r_n^{\alpha}}\Big).
	\end{equation*}
	Due to Markov's inequality and the second assumption (\ref{releq:convrategen2}) this term is finally not bigger than \begin{align*}
	&\sum_{\substack{\scriptscriptstyle j\geq M \\ \scriptscriptstyle 2^j\leq \rho\,r_n } } \frac{r_n^{\alpha}}{C \sqrt{n} 2^{(j-1)\alpha}} \mathbb{E}^*\Big[\sup_{d_0(\eta,\eta_0)\leq \rho}\sup_{d_1(\vartheta, \vartheta_0)\leq 2^j/r_n } \big|\sqrt{n}(\mathbb{E}_n-\mathbb{E})\big[m_{\eta, \vartheta}(Y) - m_{\eta, \vartheta_0}(Y)\big]\big|\Big]\\
	&\leq \sum_{\substack{\scriptscriptstyle j\geq M \\ \scriptscriptstyle 2^j\leq \rho\,r_n} } \frac{\big(\nicefrac{2^j}{r_n}\big)^{\beta}r_n^{\alpha}}{ \sqrt{n} 2^{(j-1)\alpha}} = 2^{\alpha} \sum_{\substack{\scriptscriptstyle j\geq M \\ \scriptscriptstyle 2^j\leq \rho\,r_n} } \frac{n^{\nicefrac{(\alpha-\beta)}{2(\alpha-\beta)}}}{\sqrt{n}2^{j(\alpha-\beta)}} = 2^{\alpha} \sum_{\substack{\scriptscriptstyle j\geq M \\ \scriptscriptstyle 2^j\leq \rho\,r_n} } \frac{1}{2^{j(\alpha-\beta)}} \leq 2^{\alpha} \sum_{j\geq M} \frac{1}{2^{j(\alpha-\beta)}}.
	\end{align*}	
	The last series can be made small by taking $M$ big enough since $\alpha>\beta$. Hence every summand in (\ref{releq:every_is_small}) can be made small and thus the theorem is proven.  
\end{proof}

\bibliographystyle{chicago}
\nocite{*}
\bibliography{supp}